\documentclass[a4paper,11pt]{article}

\usepackage{amsfonts}
\usepackage{amssymb}
\usepackage{amsthm}
\usepackage{amsmath}
\usepackage{a4wide}
\usepackage{mathrsfs}
\usepackage{epsfig}
\usepackage{palatino}
\usepackage{tikz,fp,ifthen}
\usepackage{float}
\usepackage{nicefrac}

\newcommand{\pdfgraphics}{\ifpdf\DeclareGraphicsExtensions{.pdf,.jpg}\else\fi}
\usepackage{graphicx}

\usepackage{color}
\definecolor{hanblue}{rgb}{0.27, 0.42, 0.81}
\definecolor{red}{rgb}{1.0, 0.0, 0.0}
\usepackage[colorlinks, citecolor=hanblue,linkcolor=red]{hyperref}

\usepackage{mathtools}
\mathtoolsset{showonlyrefs}

\numberwithin{equation}{section}

\theoremstyle{plain}

\newtheorem{thm}{Theorem}[section]
\newtheorem{lem}[thm]{Lemma}
\newtheorem{prop}[thm]{Proposition}
\newtheorem{cor}[thm]{Corollary}

\theoremstyle{definition}
\newtheorem{defn}[thm]{Definition}

\theoremstyle{remark}
\newtheorem{rem}[thm]{Remark}

\numberwithin{equation}{section}

\def\R{\mathbb R}

\newcommand{\intbar}{\etaathop{\int\etaakebox(-13.5,0){\rule[4pt]{.7em}{0.3pt}}
\kern-6pt}\nolimits}

\newcommand{\be}{\begin{equation}}
\newcommand{\ee}{\end{equation}}
\newcommand{\bea}{\begin{equation*}}
\newcommand{\eea}{\end{equation*}}

\newcommand{\vertiii}[1]{{\left\vert\kern-0.25ex\left\vert\kern-0.25ex\left\vert #1 
		\right\vert\kern-0.25ex\right\vert\kern-0.25ex\right\vert}}
\newcommand{\vertiiii}[1]{{\left\vert\kern-0.25ex\left\vert\kern-0.25ex\left\vert #1 
		\right\vert\kern-0.25ex\right\vert\kern-0.25ex\right\vert}}

\def\R{{{\mathbb R}}}

\newcommand{\N}{\mathbb{N}}

\def\be{\begin{equation}}
\def\ee{\end{equation}}
\def\bea{\begin{eqnarray*}}
\def\bean{\begin{eqnarray}}
\def\eean{\end{eqnarray}}
\def\eea{\end{eqnarray*}}

\begin{document}
\pdfgraphics 

\title{Existence and Uniqueness of the 
Motion by Curvature of Regular Networks}

\author{Michael G\"{o}{\ss}wein
\footnote{Fakult\"at f\"ur Mathematik, Universit\"at Duisburg--Essen,
Thea-Leymann-Stra{\ss}e 9,
45127 Essen, Germany}
\and Julia Menzel  \footnote{Fakult\"at f\"ur Mathematik, Universit\"at Regensburg,
Universit\"atsstrasse 31, 93053 Regensburg, Germany}
\and Alessandra Pluda
\footnote{Dipartimento di Matematica, Universit\`a di Pisa, Largo
    Bruno Pontecorvo 5, 56127 Pisa, Italy}}
    
\date{}

\maketitle

\begin{abstract}
\noindent We prove existence and uniqueness of the motion by curvature
of networks with triple junctions in $\mathbb{R}^n$ when the initial datum is of class
$W^{2-\nicefrac{2}{p}}_p$ and the unit tangent
vectors to the concurring curves form angles of $120$ degrees.
Moreover we investigate the regularisation 
effect due to the parabolic nature of the system.
An application of this wellposedness result is a new proof
of~\cite[Theorem~3.18]{mannovtor} where the possible behaviours
of the solutions at the maximal time of existence are described. \\
Our study is motivated by an open question proposed in~\cite{mannovplusch}:
does there exist a unique solution of the motion by curvature of networks
with initial datum being a regular network of class $C^2$? We give a positive answer.
\end{abstract}




\section{Introduction}

The mean curvature flow of surfaces in $\mathbb{R}^n$,
and in Riemannian manifolds in general, 
is one of the most significant examples of 
geometric evolution equations.
This evolution can be understood
as the gradient flow of the area functional:
a time--dependent surface evolves with normal velocity equal to its mean curvature
at any point and time.

From the 80s  
the curve shortening flow (mean curvature flow of 
one--dimensional objects)
was widely studied by many authors
both for closed curves~\cite{gage0,gage,gaha1,gray1}
and for curves with fixed end--points~\cite{huisk2,stahl1,stahl2}.
Also two concurring curves forming an angle or a cusp
can be regarded as a single curve with a singular point
which will vanish immediately under the flow~\cite{angen3,angen1,angen2}.
When more than two curves meet at a junction, the description 
of the motion cannot be reduced to the case of a single curve
and the problem presents new interesting features.
The simplest example of motion by mean curvature of a set which is essentially singular 
is indeed the motion by curvature of \emph{networks} that are
finite unions of curves that meet at junctions. 

Although after the seminal work by Brakke~\cite{brakke}
several weak definitions of the motion by curvature 
of singular surfaces have been proposed, 
the first attempt to 
find strong solutions to the network flow was
by Bronsard and Reitich~\cite{Bronsardreitich} providing a well posedness result 
for initial data of class $C^{2+\alpha}$.
The analysis of the long time behavior of 
the evolving networks was undertaken 
in~\cite{mannovtor}, completed in~\cite{MMN13} for trees composed of three curves 
and extended to more general cases in~\cite{IlmNevSch,mannovplusch,pluda}. 

In this paper we restrict to \emph{regular} networks
that possess only triple junctions where the unit tangent vectors
of the concurring curves form angles of $120$ degrees.
The motion by curvature of networks can be expressed as a boundary value problem
where the evolution of each curve is described by a second order quasilinear PDE
as given in Definition~\ref{geosolution}.

Our main result is the following.

\begin{thm}[Existence, uniqueness and smoothness of the motion by curvature]\label{existenceuniqueness}
Let $p\in (3,\infty)$ and $\mathcal{N}_0$ be a regular network in $\mathbb{R}^n$
of class $W^{2-\nicefrac{2}{p}}_p$. Then there exists a maximal solution 
$\left(\mathcal{N}(t)\right)_{t\in [0, T_{\max})}$ 
to the motion by curvature with initial datum $\mathcal{N}_0$ in the maximal time interval $[0,T_{\max})$ which is geometrically unique and locally of regularity
$$
\boldsymbol{E}_T:=W_p^1\left((a_n,b_n);L_p((0,1);(\mathbb{R}^n)^3)\right)\cap L_p\left((0,T);
W_p^2\left((a_n,b_n);(\mathbb{R}^n)^3\right)\right)
$$ 
for intervals $[a_n,b_n]\subset[0,T_{max})$, $n\in\mathbb{N}$, with $\bigcup_{n\in\mathbb{N}}[a_n,b_n]=[0,T_{max})$. Furthermore, the parametrisation $\gamma:[0,T_{max})\times[0,1]\to(\mathbb{R}^n)^3$ that parametrises the curves of the networks with constant speed equal to their length is smooth for positive times.
\end{thm}

As several solutions to the motion by curvature 
can be obtained by parametrising the same set with different maps, 
the uniqueness has to be understood in a purely geometric sense namely,
up to reparametrisations.

Local existence and uniqueness was proved in~\cite{Bronsardreitich}
for admissible initial networks of class
$C^{2+\alpha}$ with the sum of the curvature at the junctions equal to zero.
When the initial datum is a 
regular network of class $C^2$ without any restriction on the curvature at the junctions,
existence (but not uniqueness) has been established 
in~\cite[Theorem 6.8]{mannovplusch}. 
Theorem~\ref{existenceuniqueness} improves the uniqueness result by Bronsard and Reitich 
passing from initial data in $C^{2+\alpha}$ to
$W^{2-\nicefrac{2}{p}}_p$
which gives \emph{a fortiori} uniqueness for regular networks of class $C^2$
and even of class $H^2$ (take any $p\in (3,6]$).


We discuss now in more details Theorem~\ref{existenceuniqueness}.
The motion by curvature of networks is described by a 
parabolic system of degenerate PDEs 
where only the normal movements of the curves are prescribed.
We specify a suitable tangential component of the velocity
to turn the problem into a system of non--degenerate
second order quasilinear PDEs, the so--called Special Flow
(Definition~\ref{analyticproblem}).
Then we linearise the Special Flow around the initial datum and
prove existence and uniqueness for the linearised problem in Section~\ref{linearisedcase}.
Wellposedness of the linear system follows by Solonnikov's theory~\cite{solonnikov2} 
provided that the system is parabolic and 
that the complementary conditions hold.
Both properties were already shown in~\cite{Bronsardreitich},
nevertheless we present a new and shorter proof of the complementary conditions.  
Solutions to the Special Flow are obtained 
by a contraction argument in Section~\ref{sectionspecialflow}.
The solution to the Special Flow induces a solution to the motion by curvature 
of networks. To conclude the uniqueness result it is then
enough to prove that two different solutions differ only by a reparametrisation
but they are actually the same set as shown in Section~\ref{realexistence}.
Existence and uniqueness of maximal solutions can then be deduced with standard arguments.
Given $p\in (3,\infty)$ our solution space $\boldsymbol{E}_T$
embeds into 
$C\left([0,T];C^{1+\alpha}\left([0,1];(\mathbb{R}^n)^3\right)\right)$.
This choice allows us to define the boundary conditions pointwise 
and to use the theory of~\cite{solonnikov2} for the associated linear system.
Moreover the above regularity is needed in the contraction estimates
because of the quasilinear nature of the equations.

\smallskip

Because of the parabolicity of the problem 
it is natural to ask  whether the regularity of the evolving
network increases during the flow.
We give a positive answer to this question in
Section~\ref{parabolicsmoothingresult}
proving that the flow is smooth for all positive times (see Theorem~\ref{smoothnessthm}).
The idea of the proof is based on the so called parameter trick which is due to Angenent~\cite{angen3}. 
Although this strategy has been generalized to 
several situations~\cite{lunardi1,lunardi2,Prusssimonett}, 
it should be pointed out that our system is not among the cases treated above
because of the fully non-linear boundary condition
\begin{equation*}
\sum_{i=1}^{3} \frac{\gamma^i_x}{\vert\gamma^i_x\vert}=0\,.
\end{equation*}
In~\cite{michi} a strategy
has been developed to prove smoothness for positive times of
the surface diffusion flow for triple junction clusters with the same non--linear
boundary condition.
We follow this approach and modify the arguments to our setting to complete the proof of Theorem~\ref{existenceuniqueness}.

\smallskip

Finally a description of the possible different
behaviours of the solutions as time tends to the maximal 
time of existence is desirable.
Thanks to Theorem~\ref{existenceuniqueness} and the quantification of the existence time of solutions to the Special Flow in terms of the initial values as given in Theorem~\ref{short time existence}
we are also able to prove the following:

\begin{thm}[Long time behaviour]\label{longtimebehavior}
Let $p\in (3,6]$, $\mathcal{N}_0$ an admissible network of class $W_p^{2-\nicefrac{2}{p}}$
and $\left(\mathcal{N}(t)\right)_{t\in [0, T_{\max})}$
be a maximal solution to
the motion by curvature with initial datum 
$\mathcal{N}_0$ in $[0,T_{\max})$ where $T_{\max}\in(0,\infty)\cup\{\infty\}$.
Then 
$$
T_{\max}=\infty
$$ 
or as $t\nearrow T_{\max}$ at least one of the following happens:
\begin{itemize}
\item[$i)$] the inferior limit of the length of at least one curve of the network
$\mathcal{N}(t)$ is zero;
\item [$ii)$] the superior limit of the $L^2$--norm of the
curvature of the network is $+\infty$.
\end{itemize}
\end{thm}

This result was first shown for planar networks in~\cite[Theorem~3.18]{mannovtor}.
The benefit of our proof is that energy estimates are completely avoided.

\medskip

We describe here the structure of the paper.
In Section~\ref{preliminaries} we define the motion by curvature of networks and
we introduce the solution space together with useful properties.
Section~\ref{existence and uniquess} is devoted to prove existence of solutions to the motion by curvature and their geometric uniqueness. Then in Section~\ref{parabolicsmoothingresult}
we explore the regularisation effect of the flow resulting in the proof of Theorem~\ref{existenceuniqueness}.
We conclude with the proof of Theorem~\ref{longtimebehavior} in Section~\ref{longtime} giving a description of the behaviour of solutions at their maximal time of existence.

\medskip

\medskip

\medskip

\textbf{Acknowledgements}

The authors gratefully acknowledge the support by the Deutsche Forschungsgemeinschaft (DFG) via the GRK 1692 ``Curvature, Cycles, and Cohomology''. 

\newpage

\section{Solutions to the Motion by Curvature of Networks}\label{preliminaries}

\subsection{Preliminaries on function spaces}

This paper is devoted to show well-posedness of a second order evolution equation. One natural solution space is given by
\begin{equation*}
W_p^{1,2}\left((0,T)\times(0,1);\mathbb{R}^d\right):=W_p^1\left((0,T);L_p((0,1);\mathbb{R}^d)\right)\cap L_p\left((0,T);W_p^2\left((0,1);\mathbb{R}^d\right)\right)
\end{equation*}
where $T$ is positive representing the time of existence and $d\in\mathbb{N}$ is any natural number.
This space should be understood as the intersection of two \textit{Bochner spaces} that are Sobolev spaces defined on a measure space with values in a Banach space. We give a brief summary in the case that the measure space is an interval. A detailed introduction on Bochner spaces can be found in~\cite{yosida}.
\begin{defn}
	Let $I\subset\mathbb{R}$ be an open interval and $X$ be a Banach space. A function $f:I\to X$ is called \textit{strongly measurable} if there exists a family of simple functions $f_n:I\to X$, $n\in\mathbb{N}$, such that for almost every $x\in I$,
	\begin{equation*}
	\lim\limits_{n\to\infty}\left\lVert f_n(x)-f(x)\right\rVert =0\,.
	\end{equation*}
	Here, a function $g:I\to X$ is called \textit{simple} if 
	\begin{equation*}
	g= \sum_{k=1}^N a_k\chi_{(b_k,c_k)}
	\end{equation*}
	for $N\in\mathbb{N}$, $a_k\in X$, $b_k, c_k\in I$ and $b_k<c_k$ for $k\in\{1,\dots,N\}$.
\end{defn}
If $f:I\to X$ is strongly measurable, then $\left\lVert f\right\rVert_X:I\to\mathbb{R}$ is Lebesgue measurable. This justifies the following definition.
\begin{defn}[$L_p$--spaces]
	Let $I\subset\mathbb{R}$ be an open interval and $X$ be a Banach space. For $1\leq p\leq \infty$, we define the \textit{$L_p$--space}
	\begin{equation*}
	L_p\left(I;X\right):=\left\{f:I\to X \text{ strongly measurable }:\left\lVert f\right\rVert_{L_p(I;X)}<\infty\right\}\,,
	\end{equation*}
	where $\left\lVert f\right\rVert_{ L_p\left(I;X\right)}:=\left\lVert \left\lVert f(\cdot)\right\rVert_X\right\rVert_{L_p\left(I;\mathbb{R}\right)}$. Furthermore, we let
	\begin{equation*}
	L_{1,loc}\left(I;X\right):=\left\{f:I\to X\text{ strongly measurable}:\text{ for all }K\subset I \text{ compact, } f_{|K}\in L_1\left(K;X\right)\right\}\,.
	\end{equation*}
\end{defn}
\begin{defn}\label{distributional derivative}
	Let $I\subset\mathbb{R}$ be an open interval, $X$ be a Banach space, $f\in L_{1,loc}(I;X)$ and $k\in\mathbb{N}_0$. The $k$-th distributional derivative $\partial_x^k f$ of $f$ is the functional on $C_0^\infty(I;\mathbb{R})$ given by
	\begin{equation*}
	\langle \phi\,,\,\partial_x^k f\rangle:=(-1)^{\lvert\alpha\rvert}\int_{I}^{}f(x)\partial_x^k\phi(x)\mathrm{d}x\,.
	\end{equation*}
	The distribution $\partial^k_x f$ is called regular if there exists $v\in L_{1,loc}(I;X)$ such that
	\begin{equation*}
	\langle \phi\,,\,\partial^k_x f\rangle:=\int_{I}^{}v(x)\phi(x)\mathrm{d}x\,.
	\end{equation*}
	In this case we write $\partial^k_x f= v\in L_{1,loc}(I;X)$.
\end{defn}
\begin{defn}[Sobolev spaces]\label{Sobolev spaces}
	Let $m\in\mathbb{N}$, $I\subset\mathbb{R}$ be an open interval and $X$ be a Banach space. For $1\leq p\leq\infty$ the \textit{Sobolev space of order $m\in\mathbb{N}$} is defined as 
	\begin{equation*}
	W_p^m(I;X):=\left\{f\in L_p(I;X):\partial^k_x f\in L_p(I;X)\text{ for all }1\leq k\leq m\right\}\,,
	\end{equation*}
	where $\partial^k_x f$ is the distributional derivative defined in Definition~\ref{distributional derivative}.
\end{defn}
It is well-known that the space $W^m_p\left(I;X\right)$ is a Banach space in the norm
	\begin{equation}
	\lVert f\rVert_{W^m_p(I;X)}:=\left\{\begin{array}{cl} \left(\sum_{0\leq k\leq m}^{}\lVert\partial^k_x f\rVert_{L_p(I;X)}^p\right)^{\nicefrac{1}{p}}\,, & 1\leq p<\infty\,,\\ \max_{0\leq k\leq m}\lVert\partial^k_x f\rVert_{L_\infty(I;X)}\,, & p=\infty\,. \end{array}\right.\label{Sobolev Norm}
	\end{equation}

Elements in the solution space 
\begin{equation*}
\boldsymbol{E}_T:=W_p^1\left((0,T);L_p((0,1);(\mathbb{R}^n)^3)\right)\cap L_p\left((0,T);W_p^2\left((0,1);(\mathbb{R}^n)^3\right)\right)
\end{equation*}
are thus functions $f\in L_p\left((0,T); L_p\left((0,1);(\mathbb{R}^n)^3)\right)\right)$ possessing one distributional derivative with respect to time $\partial_t f\in L_p\left((0,T);L_p\left((0,1);(\mathbb{R}^n)^3)\right)\right)$ in the sense of Definition~\ref{distributional derivative}. Furthermore, for almost every $t\in(0,T)$, the function $f(t)$ lies in $W_p^2\left((0,1);(\mathbb{R}^n)^3)\right)$ and thus has two spacial derivatives $\partial_x (f(t))$, $\partial_x ^2\left(f(t)\right)\in L_p\left((0,1);(\mathbb{R}^n)^3)\right)$. One easily sees that the functions $t\mapsto \partial_x^k(f(t))$ for $k\in\{1,2\}$ lie in $L_p\left((0,T);L_p\left((0,1);(\mathbb{R}^n)^3)\right)\right)$. The space $\boldsymbol{E}_T$ is often denoted by $W_p^{1,2}\left((0,T)\times(0,1);(\mathbb{R}^n)^3)\right)$. We also use the notation $\left\lVert\cdot\right\rVert_{W_p^{1,2}}:=\left\lVert\cdot\right\rVert_{\boldsymbol{E}_T}$ where $\left\lVert\cdot\right\rVert_{\boldsymbol{E}_T}$ is the corresponding norm on $\boldsymbol{E}_T$ as defined in~\eqref{Sobolev Norm}.

\begin{defn}[Sobolev--Slobodeckij spaces]
	Given $d\in\mathbb{N}$, $p\in [1,\infty)$ and $\theta\in(0,1)$ the \textit{Slobodeckij semi-norm} of an element $f\in L_p\left((0,1);\mathbb{R}^d\right)$ is defined as
	\begin{equation*}
	\left[f\right]_{\theta,p}:=\left(\int_{0}^{1}\int_{0}^{1}\frac{\left\lvert f(x)-f(y)\right\rvert^p}{\lvert x-y\rvert^{\theta p+1}}\,\mathrm{d}x\,\mathrm{d}y\right)^{\nicefrac{1}{p}}\,.
	\end{equation*}
	Let $s\in(0,\infty)$ be non--integer. The \textit{Sobolev--Slobodeckij space} $W_p^s\left((0,1);\mathbb{R}^d\right)$ is defined by
	\begin{equation*}
	W_p^s\left((0,1);\mathbb{R}^d\right):=\left\{f\in W_p^{\lfloor s\rfloor}\left((0,1);\mathbb{R}^d\right):\left[\partial_x^{\lfloor s\rfloor}f\right]_{s-\lfloor s\rfloor,p}<\infty\right\}\,.
	\end{equation*}
\end{defn}
The following result characterises the regularity of the initial values.
\begin{thm}\label{embeddingBUC}
	Let $T$ be positive, $p\in(3,\infty)$ and $\alpha\in\left(0,1-\nicefrac{3}{p}\right]$. We have continuous embeddings
	\begin{equation*}
	W_p^{1,2}\left((0,T)\times(0,1)\right)\hookrightarrow C\left([0,T];W_p^{2-\nicefrac{2}{p}}\left((0,1)\right)\right)\hookrightarrow C\left([0,T];C^{1+\alpha}\left([0,1]\right)\right)\,.
	\end{equation*}
\end{thm}
\begin{proof}
	The first embedding follows from~\cite[Lemma 4.4]{DenkSaalSeiler}, the second is an immediate consequence of the Sobolev Embedding Theorem~\cite[Theorem 4.6.1.(e)]{Triebel}.
\end{proof}

Similarly, we can specify the spaces of the boundary values.
\begin{lem}\label{derivativesol}
	Let $T$ be positive, $d\in\mathbb{N}$ and $p\in [1,\infty)$. Then the operator
	\begin{align*}
	W_p^{1,2}\left((0,T)\times(0,1);\mathbb{R}\right)&\to W_p^{\nicefrac{1}{2}-\nicefrac{1}{2p}}\left((0,T);L_p\left(\{0\};\mathbb{R}\right)\right)\cap L_p\left((0,T);W_p^{1-\nicefrac{1}{p}}\left(\{0\};\mathbb{R}\right)\right)\,, \\
	f&\mapsto \left(f_{x}\right)_{|x=0}
	\end{align*}
	is linear and continuous.
\end{lem}
\begin{proof}
	This follows from~\cite[Theorem 5.1]{solonnikov2}.
\end{proof}
In this work we will use the following identification.
\begin{prop}\label{identification}
	Let $T$ be positive, $d\in\mathbb{N}$ and $p\in[1,\infty)$. There is an isometric isomorphism
		\begin{equation*}
	W_p^{\nicefrac{1}{2}-\nicefrac{1}{2p}}\left((0,T);L_p\left(\{0\};\mathbb{R}^d\right)\right)
	\cap L_p\left((0,T); W_p^{1-\nicefrac{1}{p}}\left(\{0\};\mathbb{R}^d\right)\right) 
	\simeq W_p^{\nicefrac{1}{2}-\nicefrac{1}{2p}}\left((0,T);\mathbb{R}^d\right)
	\end{equation*}
		via the map $f\mapsto \left(t\mapsto f(t,0)\right)$. 
\end{prop}
\begin{proof}
	It is shown in~\cite[page~406]{Lee} that integration with respect to the volume element on the $0$--dimensional manifold $\{0\}$ is  given by integration with respect to the counting measure. That allows us to identify the space 
	\begin{equation*}
	L^p\left(\{0\};\mathbb{R}^d\right)
	:=\left\{f:\{0\}\to\mathbb{R}^d:\int_{\{0\}}^{}\left\lvert f\right\rvert^p\,\mathrm{d}\sigma
	=\left\lvert f(0)\right\rvert^p <\infty\right\}
	\end{equation*}
	with $\mathbb{R}^d$ via the isometric isomorphism 
	$I:L_p\left(\{0\};\mathbb{R}^d\right)\to\mathbb{R}^d$, $f\mapsto f(0)$. One easily sees that this operator restricts to $I:W_p^s\left(\{0\};\mathbb{R}^d\right)\to \mathbb{R}^d$ for every $s>0$.
\end{proof}

Another important feature of Sobolev Slobodeckij spaces is their Banach algebra property.

\begin{prop}\label{Banachalgebra}
	Let $I\subset\mathbb{R}$ be a bounded open interval, $p\in[1,\infty)$ and $s\in(0,1)$ with $s-\frac{1}{p}>0$. Then for $f,g\in W_p^s\left(I;\mathbb{R}\right)$ the product $fg$ lies in $W_p^s\left(I;\mathbb{R}\right)$ and satisfies
	\begin{equation*}
	\left\lVert fg\right\rVert_{W_p^s(I;\mathbb{R})}\leq C(s,p)\left(\left\lVert f\right\rVert_{C(\overline{I})}\left\lVert g\right\rVert_{W_p^s(I;\mathbb{R})}+\left\lVert g\right\rVert_{C(\overline{I})}\left\lVert f\right\rVert_{W_p^s(I;\mathbb{R})}\right)\,.
	\end{equation*}
	Furthermore, given a smooth function $F:\mathbb{R}^d\to\mathbb{R}$, $d\in\mathbb{N}$, and a function $f\in W_p^s\left(I;\mathbb{R}^d\right)$, the function $t\mapsto F(f(t))$ lies in $W_p^s\left(I;\mathbb{R}\right)$.
\end{prop}
\begin{proof}
	As $W_p^s\left((0,1);\mathbb{R}\right)\hookrightarrow C(\overline{I};\mathbb{R})$ due to the Sobolev Embedding Theorem~\cite[Theorem 4.6.1.(e)]{Triebel}, we obtain for $f,g\in W_p^s\left(I;\mathbb{R}\right)$ the estimate
	\begin{equation*}
	\left\lVert fg\right\rVert_{L_p(I;\mathbb{R})}\leq\left\lVert f\right\rVert_{C(\overline{I})}\left\lVert g\right\rVert_{L_p(I;\mathbb{R})}\leq C(s,p)\lVert f\rVert_{W_p^s(I;\mathbb{R})}\left\lVert g\right\rVert_{W_p^s(I;\mathbb{R})}
	\end{equation*}
	and
	\begin{align*}
	[fg]^p_{s,p}&=\int_I\int_I\frac{\vert(fg)(x)-(fg)(y)\vert^p}{\vert x-y\vert^{sp+1}}\mathrm{d}x\,\mathrm{d}y\\
	&\leq \int_I\int_I\frac{\vert g(x)\vert^p\vert f(x)-f(y)\vert^p+\vert f(y)\vert^p\vert g(x)-g(y)\vert^p}{\vert x-y\vert^{sp+1}}\mathrm{d}x\,\mathrm{d}y\\
	&\leq \left\lVert g\right\rVert_{ C(\overline{I})}^p[f]^p_{s,p}+\left\lVert f\right\rVert_{ C(\overline{I})}[g]_{s,p}^p\leq C(s,p)\lVert f\rVert_{W_p^s(I;\mathbb{R})}\left\lVert g\right\rVert_{W_p^s(I;\mathbb{R})}\,.
	\end{align*}
	Let $F:\mathbb{R}^d\to\mathbb{R}$ be smooth and $f\in W_p^s\left(I;\mathbb{R}^d\right)$. As $f$ lies in $C(\overline{I};\mathbb{R}^d)$, there exists $R>0$ such that $f(\overline{I})\subset\overline{B_R(0)}$. Thus we obtain
	\begin{equation*}
	\left\lVert F(f)\right\rVert^p_{L_p(I;\mathbb{R})}=\int_I\vert F(f(x))\vert^p\mathrm{d}x\leq \max_{z\in \overline{B_R(0)}}\vert F(z)\vert^p \vert I\vert
	\end{equation*}
	where $\vert I\vert$ denotes the length of the interval $I$. Using
	\begin{align*}
	\vert F(f(x))-F(f(y))\vert&=\left\vert\int_0^1 (DF)\left(\xi f(x)+(1-\xi)f(y)\right)\mathrm{d}\xi\,(f(x)-f(y))\right\vert\\
	&\leq \max_{z\in\overline{B_R(0)}}\vert DF(z)\vert \vert f(x)-f(y)\vert
	\end{align*}
	we obtain
	\begin{align*}
	[F(f)]^p_{s,p}=\int_I\int_I\frac{\vert F(f(x))-F(f(y))\vert ^p}{\vert x-y\vert^{sp+1}}\mathrm{d}x\,\mathrm{d}y\leq [f]_{s,p}^p\max_{z\in \overline{B_R(0)}}\vert DF(z)\vert^p\,.
	\end{align*}
\end{proof}

To show well-posedness of evolution equations it is important to have embeddings with constants independent of the time interval one is working with. To this end one needs to change the norm on the solution space. In the following, we collect the results that are needed in our specific case.
\begin{cor}\label{equivalent norm $E_T$}
	Let $p\in(3,\infty)$. For every $T>0$,
	\begin{equation*}
	\vertiii{g}_{W_p^{1,2}\left((0,T)\times(0,1)\right)}:=\left\lVert g\right\rVert_{W_p^{1,2}\left((0,T)\times(0,1)\right)}+\left\lVert g(0)\right\rVert_{W_p^{2-\nicefrac{2}{p}}\left((0,1)\right)}
	\end{equation*}
	defines a norm on $W_p^{1,2}\left((0,T)\times(0,1)\right)$ that is equivalent to the usual one.
\end{cor}
\begin{proof}
	This is a consequence of Theorem~\ref{embeddingBUC}.
\end{proof}

\begin{lem}\label{extensionE_T}
	Let $T_0$ be positive, $T\in(0,T_0)$ and $p\in(3,\infty)$. There exists a linear operator
	\begin{equation*}
	\boldsymbol{E}: W_p^{1,2}\left((0,T)\times(0,1)\right)\to W_p^{1,2}\left((0,T_0)\times(0,1)\right)
	\end{equation*}
	such that for all $g\in W_p^{1,2}\left((0,T)\times(0,1)\right)$, $\left(\boldsymbol{E}g\right)_{|(0,T)}=g$ and
	\begin{equation*}
	\left\lVert\boldsymbol{E}g\right\rVert_{W_p^{1,2}\left((0,T_0)\times(0,1)\right)}\leq C\left(\left\lVert  g\right\rVert_{W_p^{1,2}\left((0,T)\times(0,1)\right)}+\left\lVert  g(0)\right\rVert_{W_p^{2-\nicefrac{2}{p}}(0,1)}\right)=C\vertiii{g}_{W_p^{1,2}\left((0,T)\times(0,1)\right)}
	\end{equation*}
	with a constant $C=C(p,T_0)$ depending only on $p$ and $T_0$.
\end{lem}
\begin{proof}
	In the case that $g(0)=0$, the function $g$ can be extended to $(0,\infty)$ by reflecting it with respect to the axis $t=T$. The general statement can be deduced from this case by solving a linear parabolic equation of second order and using results on maximal regularity as given in~\cite[Proposition 3.4.3]{Prusssimonett}.
\end{proof}
Given $d\in\mathbb{N}$ we obtain an extension operator on the space $W_p^{1,2}\left((0,T)\times(0,1);\mathbb{R}^d\right)$
by applying $\boldsymbol{E}$ to every component.

\begin{lem}\label{equivnorm2}
	Let $p\in(1,\infty)$ and $\alpha>\frac{1}{p}$. For every positive $T$,
	\begin{equation*}
	\vertiii{b}_{W_p^{\alpha}\left((0,T);\mathbb{R}\right)}:=\left\lVert b\right\rVert_{W_p^{\alpha}\left((0,T);\mathbb{R}\right)}+\vert b(0)\vert
	\end{equation*}
	defines a norm on $W_p^{\alpha}\left((0,T);\mathbb{R}\right)$ that is equivalent to the usual one.
\end{lem}
\begin{proof}
	 This is an immediate consequence of the Sobolev Embedding Theorem~\cite[Theorem 4.6.1.(e)]{Triebel}. 
\end{proof}
\begin{lem}\label{extensionboundarydata}
	Let $T$ be positive, $p\in(1,\infty)$ and $\alpha>\frac{1}{p}$. There exists a linear operator
	\begin{equation*}
	E:W_p^{\alpha}\left((0,T);\mathbb{R}\right)\to W_p^{\alpha}\left((0,\infty);\mathbb{R}\right)
	\end{equation*}
	such that for all $b\in W_p^{\alpha}\left((0,T);\mathbb{R}\right)$, $\left(Eb\right)_{|(0,T)}=b$ and
	\begin{equation*}
	\left\lVert Eb\right\rVert_{W_p^{\alpha}\left((0,\infty);\mathbb{R}\right)}\leq C_p\left(\left\lVert b\right\rVert_{W_p^{\alpha}\left((0,T);\mathbb{R}\right)}+\lvert b(0)\rvert\right)=C_p	\vertiii{b}_{W_p^{\alpha}\left((0,T);\mathbb{R}\right)}
	\end{equation*}
	with a constant $C_p$ depending only on $p$.
\end{lem}
\begin{proof}
	In the case $b(0)=0$ the operator obtained by reflecting the function with respect to the axis $t=T$ has the desired properties. The general statement can be deduced from this case using surjectivity of the temporal trace
	$
	_{|t=0}:W_p^{\alpha}\left((0,\infty);\mathbb{R}\right)\to\mathbb{R}\,.
	$
\end{proof}

\begin{thm}[Uniform embedding I]\label{embeddingBUCimeindependent}
	Let $p\in(3,\infty)$ and $T_0$ be positive. There exist constants $C(p)$ and $C\left(T_0,p\right)$ such that for all $T\in (0,T_0]$ and all $g\in W_p^{1,2}\left((0,T)\times(0,1)\right)$,
	\begin{equation*}
	\left\lVert g\right\rVert_{C\left([0,T];C^1\left([0,1]\right)\right)}\leq C(p)\left\lVert g\right\rVert_{C\left([0,T];W_p^{2-\nicefrac{2}{p}}\left((0,1)\right)\right)}\leq C\left(T_0,p\right)\vertiii{g}_{W_p^{1,2}\left((0,T)\times(0,1)\right)}\,.
	\end{equation*}
\end{thm}
\begin{proof}
	Let $T\in (0,T_0]$ be arbitrary, $g\in W_p^{1,2}\left((0,T)\times(0,1)\right)$ and $\boldsymbol{E}g$ the extension according to Lemma~\ref{extensionE_T}. Then $\boldsymbol{E}g$ lies in $W_p^{1,2}\left(\left(0,T_0\right)\times(0,1)\right)$ and Theorem~\ref{embeddingBUC} and Lemma~\ref{extensionE_T} imply
	\begin{align*}
	\left\lVert g\right\rVert_{C\left([0,T];W_p^{2-\nicefrac{2}{p}}\left((0,1)\right)\right)}&\leq \left\lVert \boldsymbol{E}g\right\rVert_{C\left(\left[0,T_0\right];W_p^{2-\nicefrac{2}{p}}\left((0,1)\right)\right)}\leq C\left(T_0,p\right)\left\lVert\boldsymbol{E}g\right\rVert_{W_p^{1,2}\left(\left(0,T_0\right)\times(0,1)\right)}\\
	&\leq C\left(T_0,p\right)\vertiii{g}_{W_p^{1,2}\left((0,T)\times(0,1)\right)}\,.
	\end{align*}
\end{proof}

\begin{thm}[Uniform embedding II]\label{uniformcalphac1}
	Let $p\in(3,\infty)$, $\theta\in\left(\frac{1+\nicefrac{1}{p}}{2-\nicefrac{2}{p}},1\right)$, $\delta\in\left(0,1-\nicefrac{1}{p}\right)$ and $T_0$ be positive. There exists a constant $C\left(T_0,p,\theta,\delta\right)>0$ such that for all $T\in(0,T_0]$ there holds the embedding 
	\begin{equation*}
	W_p^{1,2}\left((0,T)\times(0,1)\right)\hookrightarrow C^{(1-\theta)\left(1-\nicefrac{1}{p}-\delta\right)}\left([0,T];C^1\left([0,1]\right)\right)
	\end{equation*}
	and all $g\in W_p^{1,2}\left((0,T)\times(0,1)\right)$ satisfy the uniform estimate
	\begin{equation*}
	\left\lVert g\right\rVert_{C^{(1-\theta)\left(1-\nicefrac{1}{p}-\delta\right)}\left([0,T];C^1\left([0,1]\right)\right)}\leq C\left(T_0,p,\theta,\delta\right)\vertiii{g}_{W_p^{1,2}\left((0,T)\times(0,1)\right)}\,.
	\end{equation*}
\end{thm}
\begin{proof}
By~\cite[Corollary 26]{Simon} there holds for any $\delta\in\left(0,1-\nicefrac{1}{p}\right)$ the continuous embedding
	\begin{equation*}
	W_p^{1,2}\left((0,T_0)\times(0,1)\right)\hookrightarrow C^{1-\nicefrac{1}{p}-\delta}\left([0,T_0];L_p((0,1))\right)
	\end{equation*}
	with operator norm depending on $T_0$. Furthermore, Theorem~\ref{embeddingBUC} gives
	\begin{equation*}
	W_p^{1,2}\left((0,T_0)\times(0,1)\right)\hookrightarrow C\left([0,T_0];W_p^{2-\nicefrac{2}{p}}((0,1))\right)\,.
	\end{equation*}
	The results in~\cite{Triebel} yield that the real interpolation space satisfies
	\begin{equation*}
	W_p^{\theta(2-\nicefrac{2}{p})}((0,1))=\left(L_p((0,1)),W_p^{2-\nicefrac{2}{p}}((0,1))\right)_{\theta,p}
	\end{equation*}
	with equivalent norms. In particular, for all $f\in W_p^{\theta(2-\nicefrac{2}{p})}((0,1))$ there holds the estimate
	\begin{equation*}
	\left\lVert f\right\rVert_{W_p^{\theta(2-\nicefrac{2}{p})}((0,1))}\leq C\left\lVert f\right\rVert^{1-\theta}_{L_p((0,1))}\left\lVert f\right\rVert^\theta_{W_p^{2-\nicefrac{2}{p}}((0,1))}\,.
	\end{equation*}
	A direct computation using the above estimate shows that for all $\alpha\in(0,1)$,
	\begin{equation*}
	 C\left([0,T_0];W_p^{2-\nicefrac{2}{p}}((0,1))\right)\cap C^\alpha\left([0,T_0];L_p((0,1))\right)\hookrightarrow C^{(1-\theta)\alpha}\left([0,T_0];W_p^{\theta(2-\nicefrac{2}{p})}((0,1))\right)
	\end{equation*}
	which yields for all $\delta\in\left(0,1-\nicefrac{1}{p}\right)$ the continuous embedding
	\begin{equation*}
	W_p^{1,2}\left((0,T_0)\times(0,1)\right)\hookrightarrow C^{(1-\theta)(1-\nicefrac{1}{p}-\delta)}\left([0,T_0];W_p^{\theta(2-\nicefrac{2}{p})}((0,1))\right)\,.
	\end{equation*}
	Due to $\theta\left(2-\nicefrac{2}{p}\right)-\frac{1}{p}>1$ the Sobolev Embedding Theorem yields
	\begin{equation*}
	W_p^{1,2}\left((0,T_0)\times(0,1)\right)\hookrightarrow C^{(1-\theta)(1-\nicefrac{1}{p}-\delta)}\left([0,T_0];C^1([0,1])\right)\,.
	\end{equation*}
	The claim now follows using the extension operator $\boldsymbol{E}$ constructed in Lemma~\ref{extensionE_T} with similar arguments as in the proof of Theorem~\ref{embeddingBUCimeindependent}.
\end{proof}

\subsection{Motion by curvature of networks}

Let $n\in\mathbb{N}, n\geq 2$.
Consider a curve $\sigma:[0,1]\to\mathbb{R}^n$ of class $C^1$.
A curve is said to be \textit{regular} if $\vert \sigma_x(x) \vert\neq 0$ for every $x\in[0,1]$.
Let us denote with $s$ the arclength parameter. We remind that 
$\partial_s=\frac{\partial_x}{\vert \sigma_x\vert}$.
If a curve $\sigma$ is of class $C^1$ and regular, its
unit tangent vector is given by
${\tau}=\sigma_s=\frac{\sigma_{x}}{\left|\sigma_{x}\right|}$.
The \textit{curvature vector} of a regular $C^2$--curve $\sigma$ is defined by
 \begin{equation*}
\boldsymbol{\kappa}:=\sigma_{ss}=\tau_{s}=\frac{\sigma_{xx}}{\left|\sigma_{x}\right|^{2}} 
-\frac{\left\langle \sigma_{xx},\sigma_{x}\right\rangle \sigma_{x}}{\left|\sigma_{x}\right|^{4}}\,.
 \end{equation*}
The \textit{curvature} is given by $\kappa=\vert \tau_s\vert$.

%
%

\begin{defn}
A 
\emph{network} $\mathcal{N}$ is a connected set in 
$\mathbb{R}^n$
consisting 
of a finite union of  regular curves $\mathcal{N}^i$ 
that meet at their endpoints in junctions.
Each curve $\mathcal{N}^i$ admits a regular
$C^1$--parametrisation, namely a map 
$\gamma^i:[0,1]\to\mathbb{R}^n$ 
of class $C^1$ with $\vert\gamma^i_x\vert\neq 0$ on $[0,1]$ 
and $\gamma^i\left([0,1]\right)=\mathcal{N}^i$.
\end{defn}


Although a network is a \textit{set} by definition, we will mainly deal with its parametrisations.
It is then natural to speak about the regularity of these maps.

\begin{defn}
Let $k\in\mathbb{N}$, $k\geq 2$, and $1\leq p\leq\infty$ with $p>\frac{1}{k-1}$. A network $\mathcal{N}$ is of class $C^k$ (or $W_p^{k}$, respectively) if it admits
a regular parametrisation 
of class $C^k$ (or $W_p^{k}$, respectively).
\end{defn}

In this paper we restrict to the class of \textit{regular networks}. 

\begin{defn}
A network is called \textit{regular} if
its curves meet at triple junctions  forming equal angles.
\end{defn}

Notice that this notion is geometric in the sense that it does not depend
on the choice of the parametrisations $\sigma^i$ of the curves of the network
$\mathcal{N}$.

We define now the \textit{motion by curvature} of regular networks:
a time dependent family of regular networks evolves with normal velocity
equal to the curvature vector at any point and any time, namely
\begin{equation*}
V^i=\boldsymbol{\kappa}^i\,.
\end{equation*}
To be more precise,  given a time dependent family of curves 
$\gamma^i$, we denote by $\boldsymbol{P}^i:\mathbb{R}^n\to\mathbb{R}^n$ the projection
onto the normal space to $\gamma^i$, namely
$\boldsymbol{P}^i:=\mathrm{Id}-\gamma_s\otimes\gamma_s$. The motion equation reads as 
\begin{equation*}
\boldsymbol{P}^i\gamma^i=\boldsymbol{\kappa}^i\,.
\end{equation*}

For the sake of presentation we restrict to the motion by curvature of a \emph{Triod}.

\begin{defn}\label{Triod}
A {\em Triod} $\mathbb{T}=\bigcup_{i=1}^{3}\sigma{}^{i}([0,1])$  is a network 
composed of three regular $C^1$--curves
$\sigma^{i}:\left[0,1\right]\to\mathbb{R}^n$ that 
intersect each other at the triple junction
$O:=\sigma^1(0)=\sigma^2(0)=\sigma^3(0)$.
The other three endpoints of the curves $\sigma^i(1)$ with $i\in\{1,2,3\}$ 
coincide with three
points $P^{i}:=\sigma^{i}\left(1\right)\in\mathbb{R}^n$. The Triod is called \textit{regular} if it is a regular network.
\end{defn}

\begin{figure}[H]
\begin{center}
\begin{tikzpicture}[scale=0.9]
\draw[shift={(-2,0)}]
(-3.73,0) node[left]{$P^1$}
to[out= 50,in=180, looseness=1] (-2,0)
to[out= 60,in=180, looseness=1.5] (-0.45,1.55)
(-2,0)
to[out= -60,in=180, looseness=0.9] (-0.75,-1.75);
\path[font=\large,shift={(-2,0)}]
(-3,0.8) node[below] {$\sigma^1$}
(-1.5,1) node[right] {$\sigma^3$}
(-0.8,-1)[left] node{$\sigma^2$}
(-2.2,0) node[below] {$O$}
(-0.21,1.35)node[above]{$\,\,\,\,\,\, P^3$}
(-0.55,-1.65) node[below] {$\,\,\,\, P^2$};
\end{tikzpicture}
\end{center}
\begin{caption}{A regular Triod in $\mathbb{R}^2$.}
\end{caption}
\end{figure}

\begin{defn}[Geometrically admissible initial Triod]\label{geomadm}
A Triod
$\mathbb{T}_0$ is a \textit{geometrically admissible initial datum}
for the motion  by curvature
if it is regular
and  each of its 
curves can be parametrised by a regular curve 
$\sigma^i\in W^{2-2/p}_p([0,1],\mathbb{R}^n)$ with $p\in(3,\infty)$. 
\end{defn}
\begin{rem}
	For $p\in(3,\infty)$ the Sobolev Embedding Theorem~\cite[Theorem 4.6.1.(e)]{Triebel} implies
	\begin{equation*}
	W_p^{2-\nicefrac{2}{p}}\left((0,1);\mathbb{R}^n\right)\hookrightarrow C^{1+\alpha}\left([0,1];\mathbb{R}^n\right)
	\end{equation*}
	for $\alpha\in\left(0,1-\nicefrac{3}{p}\right)$. In particular, any admissible initial network is of class $C^1$ and the angle condition at the boundary is well-defined.
\end{rem}

\begin{defn}[Solutions to the motion by curvature]\label{geosolution}
Let $p\in (3,\infty)$ and $T>0$.
Let $\mathbb{T}_0$ be a geometrically admissible initial Triod 
with endpoints $P^1$, $P^2$, $P^3$.
A time dependent family of Triods $\left(\mathbb{T}(t)\right)$ is a
\textit{solution to the motion by curvature in $[0,T]$ 
with initial datum $\mathbb{T}_0$ }
if there exists a collection of time dependent parametrisations
\begin{equation*}
\gamma^i_n\in  W^1_p(I_n;L_p((0,1);\mathbb{R}^n))\cap 
L_p(I_n;W^2_p((0,1);\mathbb{R}^n))\,,
\end{equation*}
with $n\in\{0,\dots, N\}$ for some $N\in\mathbb{N}$, 
$I_n:=(a_n,b_n)\subset \mathbb{R}$, $a_n\leq a_{n+1}$, $b_n\leq b_{n+1}$, $a_n<b_n$
and $\bigcup_n (a_n,b_n)=(0,T)$ such that for all $n\in\{0,\dots,N\}$ 
and $t\in I_n$, $\gamma_n(t)=\left(\gamma^1(t),\gamma^2(t),\gamma^3(t)\right)$ 
is a regular parametrisation of $\mathbb{T}(t)$.
Moreover, each $\gamma_n$ needs to satisfy the following system:
\begin{equation}\label{systemtriod}
\begin{cases}
\begin{array}{lll}
\boldsymbol{P}^i\gamma^i_t(t,x)=\boldsymbol{\kappa}^i(t,x)
\quad\quad&\text{ motion by curvature,}\\
\gamma^i(t,1)=P^i\quad
&\text{ fixed endpoints,}\\
\gamma^{1}(t,0)=\gamma^{2}(t,0)=\gamma^{3}(t,0)  &\text{ concurrency condition,}\\
\sum_{i=1}^3\tau^{i}(t,0)=0\quad
&\text{ angle condition,}
\end{array}
\end{cases}
\end{equation}
for almost every $t\in  I_n,x\in\left(0,1\right)$ and for $i\in\{1,2,3\}$. 
Finally, we ask that $\gamma_n(0,[0,1])=\mathbb{T}_0$ whenever $a_n=0$. 
\end{defn}

%

\begin{rem}
In the motion by curvature equation
only the normal component of the velocity $\gamma^i_t$ is prescribed.
This does not mean that there is no 
tangential motion. Indeed, a non--trivial tangential velocity is generally needed
to allow for motion of the triple junction.
\end{rem}

\begin{rem}
We are interested in finding a time--dependent family of networks $\left(\mathcal{N}(t)\right)$ solving the motion by curvature. Our notion of solution
allows the network to be parametrised by different sets of functions
in different (but overlapping) time intervals.
Namely a solution can be parametrised by 
$\gamma=(\gamma^1,\gamma^2,\gamma^3)$ 
with $\gamma^i:(a_0,b_0)\times [0,1]\to\mathbb{R}^n$
and $\eta=(\eta^1,\eta^2,\eta^3)$ with $\eta^i:(a_1,b_1)\times [0,1]\to\mathbb{R}^n$
if $a_0\leq a_1<b_0\leq b_1$ and $\gamma^i((a_1,b_0)\times [0,1])=\eta^i((a_1,b_0)\times [0,1])$.
Requiring that the family of networks $\left(\mathcal{N}(t)\right)$
is parametrised by \emph{one map} 
$\gamma(t)=(\gamma^1(t),\gamma^2(t),\gamma^3(t))$
in the whole time interval of existence $[0,T]$
as in~\cite{mannovplusch}
gives a slightly stronger definition of the  
motion by curvature in comparison to Definition~\ref{geosolution}. 
This difference does not affect the proof of the short time existence result, 
but in principle using 
our definition the \emph{maximal} time interval of existence
could be longer.
\end{rem}

The first step to find solutions to the motion by curvature
is to turn system~\eqref{systemtriod} 
into a system of quasilinear parabolic PDEs 
by choosing a suitable tangential velocity $T$.
We choose $T$ such that
\begin{equation}
\gamma^i_{t}(t,x)=\boldsymbol{P}^i\gamma^i_{t}(t,x)
+\left\langle \gamma^i_{t}(t,x)\,,\tau^i(t,x)\right\rangle\tau^i(t,x)
=\boldsymbol{\kappa}^i(t,x)+T^i(t,x)\tau^i(t,x)
=\frac{\gamma^i_{xx}(t,x)}{\vert \gamma^i_x(t,x)\vert^2}\,.
\end{equation}
Since the expression of the curvature reads as
\begin{equation*}
\boldsymbol{\kappa}^i(t,x)
=\frac{\gamma^i_{xx}(t,x)}{\vert \gamma^i_x(t,x)\vert^2}-\left\langle\frac{\gamma^i_{xx}(t,x)}{\vert\gamma^i_x(t,x)\vert^2}\,,\tau^i(t,x)
\right\rangle\tau^i(t,x)
\end{equation*}
we choose 
\begin{equation*}
T^i(t,x)=\left\langle\frac{\gamma^i_{xx}(t,x)}{\vert\gamma^i_x(t,x)\vert^2}\,,\tau^i(t,x)
\right\rangle\,.
\end{equation*}

The equation $\gamma^i_t=\frac{\gamma^i_{xx}}{\vert \gamma^i_x\vert^2}$
is called \emph{Special Flow}.

\begin{defn}[Admissible initial parametrisation]
Let $p\in(3,\infty)$.
An \textit{admissible initial parametrisation}
for a  Triod $\mathbb{T}_0$ is a triple
$\sigma=(\sigma^1,\sigma^2,\sigma^3)$
where $\bigcup_i \sigma^i([0,1])=\mathbb{T}_0$,
$\sigma^1(0)=\sigma^2(0)=\sigma^3(0)$ and
$\sum_{i=1}^3\frac{\sigma^i_x(0)}{\vert\sigma^i_{x}(0)\vert}=0$
with $\sigma^i$ regular and of class $W^{2-2/p}_p((0,1),\mathbb{R}^n)$.
\end{defn}
Notice that it follows by the very definition that a geometrically admissible Triod admits an
admissible parametrisation. 

\begin{defn}[Solution of the Special Flow]\label{analyticproblem}
Let $T>0$ and $p\in (3,\infty)$.
Consider an admissible initial  parametrisation
$\sigma=(\sigma^1,\sigma^2,\sigma^3)$
for a Triod $\mathbb{T}_0$
in $\mathbb{R}^{n}$ with $\sigma^i(1)=P^i\in\mathbb{R}^n$.  Then
we say that $\gamma=(\gamma^1,\gamma^2,\gamma^3)$
is a \textit{solution of the Special Flow in the time interval $[0,T]$ with initial datum $\sigma$}
if 
\begin{equation*}
\gamma=\left(\gamma^1,\gamma^2,\gamma^3\right)\in \boldsymbol{E}_T=W^1_p((0,T);L_p((0,1);(\mathbb{R}^n)^3))\cap L_p((0,T);W^2_p((0,1);(\mathbb{R}^n)^3))\,,
\end{equation*}
$\vert\gamma^i_x(t,x)\vert\neq 0$ for all $(t,x)\in[0,T]\times[0,1]$ and the following system is satisfied
for $i\in\{1,2,3\}$ and for almost every $x\in(0,1)$, $t\in(0,T)$:  
\begin{equation}\label{problema}
\begin{cases}
\begin{array}{lll}
\gamma^i_t(t,x)=\frac{\gamma_{xx}^{i}\left(t,x\right)}{\left|\gamma_{x}^{i}\left(t,x\right)\right|^{2}}
\qquad &\text{Special Flow,}\\
\gamma^i(t,1)=P^i
\qquad &\text{fixed endpoints,}\\
\gamma^1(t,0)=\gamma^2(t,0)=\gamma^3(t,0)\qquad
&\text{concurrency condition,}\\
\sum_{i=1}^{3}\frac{\gamma_{x}^{i}\left(t,0\right)}{\left|\gamma_{x}^{i}\left(t,0\right)\right|}=0
\qquad &\text{angle condition,}\\
\gamma^i(0,x)=\sigma^i(x)\qquad &\text{initial datum.}\\
\end{array}
\end{cases}
\end{equation}
\end{defn}

\begin{rem}
Both in~\cite{Bronsardreitich} and in~\cite{mannovtor} the authors 
define the motion by curvature introducing directly the Special Flow.
This is not restrictive to get a short time existence result because 
a solution of the Special Flow as defined in Definition~\ref{analyticproblem} 
induces a solution of the motion by curvature in the sense of Definition~\ref{geosolution}
which is shown in Theorem~\ref{existencegeopro} below.
However, we will see that it is not trivial to deduce \textit{geometric uniqueness}
of solutions to the motion by curvature from uniqueness of solutions to the Special Flow.
\end{rem}

\section{Existence and Uniqueness of the Motion by Curvature}\label{existence and uniquess}

\subsection{Existence and uniqueness of the linearised Special Flow}\label{linearisedcase}

We fix an admissible initial parametrisation $\sigma=(\sigma^1,\sigma^2,\sigma^3)$.
Linearising the main equation of system~\eqref{problema} around the initial datum we obtain: 

\begin{equation}\label{deff}
\gamma^i_{t}(t,x)-\frac{1}{\left|\sigma^i_{x}(x)\right|^{2}}\,\gamma^i_{xx}(t,x)=
\left(\frac{1}{\left|\gamma^i_{x}(t,x)\right|^{2}}-\frac{1}{\left|\sigma^i_x(x)\right|^{2}}\right)
\gamma^i_{xx}(t,x) \,.
\end{equation}

The linearisation of the angle condition in $x=0$ is given by
\begin{equation}\label{defpsi}
-\sum_{i=1}^3 \left(\frac{\gamma^i_x}{\vert \sigma^i_x\vert}
-\frac{\sigma^i_x\left\langle
\gamma^i_x,\sigma^i_x\right\rangle}{\vert \sigma^i_x\vert^3}\right)=
\sum_{i=1}^3 \left(\left(\frac{1}{\vert \gamma^i_x\vert}
 -\frac{1}{\vert\sigma^i_x\vert}\right)\gamma^i_x +
 \frac{\sigma^i_x\left\langle \gamma^i_x,\sigma^i_x\right\rangle}{\vert \sigma^i_x\vert^3}
\right)\,,
\end{equation}
where we have omitted the dependence of the left-hand side on $(t,0)$. 
The concurrency and the fixed endpoints conditions are already linear and affine.
We obtain the following linearised system for a general right hand side $(f,\eta,b,\psi)$.
\begin{equation}\label{linsys}
\begin{cases}
\begin{array}{rll}
\gamma^i_{t}(t,x)-\frac{1}{\left|\sigma^i_{x}(x)\right|^{2}}\,\gamma^i_{xx}(t,x)&=f^i(t,x)\,, &t\in(0,T)\,, x\in(0,1)\,,i\in\{1,2,3\}\,,\\
\gamma(t,1)&=\eta(t)\,, &t\in[0,T]\,,\\
\gamma^1\left(t,0\right)-\gamma^{2}\left(t,0\right)&=0\,, &t\in[0,T]\,,\\
\gamma^2(t,0)-\gamma^{3}\left(t,0\right)&=0\,,  &t\in[0,T]\,, \\
-\sum_{i=1}^3 \left(\frac{\gamma^i_x(t,0)}{\vert \sigma^i_x(0)\vert}
-\frac{\sigma^i_x(0)\left\langle
	\gamma^i_x(t,0),\sigma^i_x(0)\right\rangle}{\vert \sigma^i_x(0)\vert^3}\right)&=b(t)\,, &t\in[0,T]\,,\\
\gamma\left(0,x\right)&=\psi\left(x\right)  \,, &x\in[0,1]\,.
\end{array}
\end{cases}
\end{equation}

\begin{defn}[Linear compatibility conditions]\label{linearcompcond}
Let $p\in (3,\infty)$.
A function $\psi=(\psi^1,\psi^2,\psi^3)$
of class $W^{2-\nicefrac{2}{p}}_p\left((0,1);(\mathbb{R}^n)^3\right)$
satisfies the \textit{linear compatibility conditions }
for system~\eqref{linsys} 
with respect to given functions $\eta\in W_p^{1-\nicefrac{1}{2p}}((0,T);(\mathbb{R}^n)^3)$, $b\in W_p^{\nicefrac{1}{2}-\nicefrac{1}{2p}}((0,T);\mathbb{R}^n)$ if
for $i,j\in\{1,2,3\}$ it holds
$\psi^i(0)=\psi^j(0)$, $\psi^i(1)=\eta^i(0)$ and 
\begin{equation*}
-\sum_{i=1}^3 \left(\frac{\psi^i_x(0)}{\vert \sigma^i_x(0)\vert}
-\frac{\sigma^i_x(0)\left\langle
\psi^i_x(0),\sigma^i_x(0)\right\rangle}{\vert \sigma^i_x(0)\vert^3}\right)=b(0)\,.
\end{equation*}
\end{defn}

We want to show that system~\eqref{linsys}
admits a unique solution $\gamma=(\gamma^1,\gamma^2,\gamma^3)$ in 
$\boldsymbol{E}_T$.
The result follows from the classical theory for linear parabolic systems by
Solonnikov~\cite{solonnikov2} provided that the system is parabolic and 
that the \textit{complementary conditions} hold (see~\cite[p.~11]{solonnikov2}).
Both the parabolicity and the complementary (initial and boundary) conditions have been
proven in~\cite{Bronsardreitich}.
We remark that the complementary conditions at the boundary follow from the \emph{Lopatinskii--Shapiro condition} (see for instance~\cite[pages 11--15]{eidelman2}).

\begin{defn}[Lopatinskii--Shapiro condition]
Let $\lambda\in\mathbb{C}$ with $ \Re(\lambda)>0$ be arbitrary.
The Lopatinskii--Shapiro condition for system~\eqref{linsys}
is satisfied at the triple junction if 
every solution $\varrho=(\varrho^1,\varrho^2,\varrho^3)\in C^2([0,\infty),(\mathbb{C}^2)^3)$ to
\begin{equation}\label{LopatinskiiShapirosystem}
\begin{cases}
\begin{array}{rll}
\lambda \varrho^i(x)-\frac{1}{\vert\sigma^i_x(0)
\vert^2}\varrho^i_{xx}(x)&=0\,, &\; \; x\in[0,\infty)\,, \;
i\in\{1,2,3\}\,,\\
\varrho^{1}(0)-\varrho^{2}(0)&=0\,, & \\
\varrho^{2}(0)-\varrho^{3}(0)&=0\,, & \\
\sum_{i=1}^3\frac{\varrho_{x}^{i}(x)}{\left|\sigma_{x}^{i}(0)\right|} -\frac{\sigma^i_x(0)\left\langle \varrho^i_x(x),\sigma^i_x(0)\right\rangle}{\vert \sigma^i_x(0)\vert^3}
&=0 &
\end{array}
\end{cases}
\end{equation}
which satisfies $\lim_{x\to\infty}\lvert \varrho^i(x)\rvert=0$ is the trivial solution. 

\medskip 

Similarly, the Lopatinskii--Shapiro condition for system~\eqref{linsys}
is satisfied at the fixed
endpoints if every solution $\varrho=(\varrho^1,\varrho^2,\varrho^3)\in C^2([0,\infty),(\mathbb{C}^2)^3)$ to
\begin{equation}\label{LopatinskiiShapiroinone}
\begin{cases}
\begin{array}{rll}
\lambda \varrho^i(x)-\frac{1}{\vert\sigma^i_x(0) \vert^2}\varrho^i_{xx}(x)&=0\,, &\;x\in[0,\infty)\,,\; i\in\{1,2,3\}\,,\\
\varrho^{i}(0)&=0\,, &\;i\in\{1,2,3\}
\end{array}
\end{cases}
\end{equation}
which satisfies $\lim_{x\to\infty}\lvert \varrho^i(x)\rvert=0$ is the trivial solution.
\end{defn}

\begin{lem}\label{LemmaLopatinskisShapiroconditions}
The  Lopatinskii--Shapiro condition is satisfied.
\end{lem}
\begin{proof}
We first check the condition at the triple junction.
Let $\varrho$ be a solution to~\eqref{LopatinskiiShapirosystem}
satisfying $\lim_{x\to\infty}\lvert \varrho^i(x)\rvert=0$.
We multiply 
$\lambda \varrho^i(x)-\frac{1}{\vert\sigma^i_x(0) \vert^2}\varrho^i_{xx}(x)=0$
by $\vert \sigma^i_x(0)\vert \boldsymbol{P}^i\overline{\varrho}^i(x)$ 
(where with $ \boldsymbol{P}^i$ we mean here $\mathrm{Id}-\sigma^i_s(0)\otimes\sigma^i_s(0)$), 
then we integrate and sum.
Using the two conditions at the boundary we get
\begin{align}\label{testedmot}
0&=\sum_{i=1}^3\int_0^\infty \lambda\vert\sigma^i_x(0)\vert\vert \boldsymbol{P}^i(\varrho^i(x))\vert^2-\frac{1}{\vert\sigma^i_x(0)\vert}\left\langle \varrho^i_{xx}(x),  \boldsymbol{P}^i\overline{\varrho}^i(x)\right\rangle \,\mathrm{d}x\\
&=\sum_{i=1}^3\int_0^\infty \lambda\vert\sigma^i_x(0)\vert\vert \boldsymbol{P}^i(\varrho^i(x))\vert^2
+\frac{\vert \boldsymbol{P}^i(\varrho^i_x(x))\vert^2}{\vert\sigma^i_x(0)\vert}\,\mathrm{d}x
-\sum_{i=1}^3 \frac{1}{\vert\sigma^i_x(0)\vert}\left\langle \boldsymbol{P}^i\varrho^i_{x}(0),  \boldsymbol{P}^i\overline{\varrho}^i(0)\right\rangle\\
&=\sum_{i=1}^3\int_0^\infty \lambda\vert\sigma^i_x(0)\vert\vert \boldsymbol{P}^i(\varrho^i(x))\vert^2
+\frac{\vert \boldsymbol{P}^i(\varrho^i_x(x))\vert^2}{\vert\sigma^i_x(0)\vert}\,\mathrm{d}x
-\left\langle\overline{\varrho}^1(0),\sum_{i=1}^3 \boldsymbol{P}^i\left(\frac{\varrho^i_x(0)}{\vert\sigma^i_x(0)\vert}\right)\right\rangle\\
&=\sum_{i=1}^3\int_0^\infty \lambda\vert\sigma^i_x(0)\vert\vert \boldsymbol{P}^i(\varrho^i(x))\vert^2
+\frac{\vert \boldsymbol{P}^i(\varrho^i_x(x))\vert^2}{\vert\sigma^i_x(0)\vert}\,\mathrm{d}x\,.
\end{align}
As a consequence we get that $\boldsymbol{P}^i(\varrho^i(x))=0$ for all $x\in[0,\infty)$ and $i\in\{1,2,3\}$
and in particular $\boldsymbol{P}^i(\varrho^1(0))=0$ for all $i\in\{1,2,3\}$. 
As the orthogonal complements of $\sigma_x^i(0)$ with $i\in\{1,2,3\}$
span all $\mathbb{R}^n$,
we conclude that $\varrho^i(0)=0$ for all $i\in\{1,2,3\}$. 
Repeating the argument and
testing the motion equation by $\vert\sigma^i_x(0)\vert\langle\overline{\varrho}^i(x),\sigma^i_s(0)\rangle\sigma^i_s(0)$
we can conclude that $\varrho^i(x)=0$ for every $x\in[0,\infty)$.
Indeed,
we obtain
\begin{align}\label{testedmot2}
&\sum_{i=1}^3\lambda\vert\sigma^i_x(0)\vert\int_0^\infty \vert \left\langle\varrho^i(x),\sigma^i_s(0)\right\rangle\vert^2\,\mathrm{d}x
+\sum_{i=1}^3\frac{1}{\vert\sigma^i_x(0)\vert}\int_0^\infty \vert \left\langle\varrho^i_x(x),\sigma^i_s(0)\right\rangle\vert^2\,\mathrm{d}x\nonumber\\
+&\sum_{i=1}^3\frac{1}{\vert\sigma^i_x(0)\vert}
\left\langle \overline{\varrho}^i(0),  \sigma^i_s(0)\right\rangle
\left\langle \varrho^i_{x}(0), \sigma^i_s(0)\right\rangle=0\,.
\end{align}
This time the boundary condition vanishes since we get 
$\varrho^i(0)=0$
from the previous step.
Taking again the real part of~\eqref{testedmot2}
we can  
conclude that $\left\langle\varrho^i(x),\sigma^i_s(0)\right\rangle=0$ for all $x\in [0,\infty)$.
Hence
$\varrho^i(x)=0$ for every $x\in[0,\infty)$ as desired.

The condition at the fixed endpoints follows in exactly the same way using 
the boundary condition $\varrho^i(0)=0$.
\end{proof}

Given $T>0$ we introduce the spaces
 \begin{itemize}
\item[] $\mathbb{E}_T:=\left\lbrace\gamma\in \boldsymbol{E}_T\,,
 \gamma^1(t,0)=\gamma^2(t,0)=\gamma^3(t,0)
\text{
 for} \;
i\in\{1,2,3\},t\in[0,T] 
\,
 \right\rbrace$,
\item[] $\mathbb{F}_T:=
\left\lbrace(f,\eta,0,b,\psi)\;\text{with}\,
f\in  L_p((0,T);L_p((0,1);(\mathbb{R}^n)^3)), \,\eta\in W^{1-\nicefrac{1}{2p}}_p((0,T);(\mathbb{R}^n)^3)\,,\right.\\
\left.\qquad\quad\;\, 0\in W^{1-\nicefrac{1}{2p}}_p((0,T);\mathbb{R}^{2n})\,, b\in W_p^{\nicefrac{1}{2}-\nicefrac{1}{2p}}((0,T);\mathbb{R}^n)\,,\,
\psi\in W_p^{2-2/p}((0,1);(\mathbb{R}^n)^3)  \right.\\
	\left.\qquad\quad\;\text{ such that the linear compatibility conditions in Definition~\ref{linearcompcond} hold}\right\rbrace$.
\end{itemize}

\begin{thm}\label{exlin}
Let $p\in(3,\infty)$.
For every $T>0$
system~\eqref{linsys} 
has a unique solution $\gamma\in\mathbb{E}_T$
provided that 
$f\in L_p((0,T);L_p((0,1);(\mathbb{R}^n)^3))$, $\eta\in W^{1-\nicefrac{1}{2p}}_p((0,T);(\mathbb{R}^n)^3)$ 
$b\in W_p^{\nicefrac{1}{2}-\nicefrac{1}{2p}}((0,T);\mathbb{R}^n)$ and
$\psi\in  W_p^{2-\nicefrac{2}{p}}((0,1);(\mathbb{R}^n)^3)$ 
fulfil the  linear compatibility conditions given in Definition~\ref{linearcompcond}.
Moreover, there exists a constant $C=C(T)>0$ such that the following
estimate holds:
\begin{equation}\label{estimate}
\Vert \gamma\Vert_{\boldsymbol{E}_T} \leq C\left( 
 \Vert f\Vert_{L_p((0,T);L_p((0,1)))}+\Vert\eta\Vert_{W^{1-\nicefrac{1}{2p}}_p((0,T))}+
\Vert b\Vert_{W_p^{\nicefrac{1}{2}-\nicefrac{1}{2p}}((0,T))}
+ \Vert \psi \Vert_{W_p^{2-\nicefrac{2}{p}}((0,1))}\right)\,.
\end{equation} 
\end{thm}
\begin{proof}
	This follows from~\cite[Theorem 5.4]{solonnikov2}.
\end{proof}

Theorem~\ref{exlin} implies in particular that
the linear operator 
$L_{T}:\mathbb{E}_T\to 
\mathbb{F}_T$ defined by
\begin{equation*}
L_{T}(\gamma)=
\begin{pmatrix}
\left(\gamma^i_t-\frac{\gamma^i_{xx}}{\vert\sigma^i_x\vert^2}\right)_{i\in\{1,2,3\}}\\
\gamma_{|x=1}\\
\left(\gamma^1_{|x=0}-\gamma^2_{|x=0},\gamma^2_{|x=0}-\gamma^3_{|x=0}\right) \\
-\sum_{i=1}^3 \left(\frac{\gamma^i_x}{\vert \sigma^i_x\vert}
-\frac{\sigma^i_x\left\langle
	\gamma^i_x,\sigma^i_x\right\rangle}{\vert \sigma^i_x\vert^3}\right)_{|x=0}\\
\gamma_{| t=0}
\end{pmatrix}
\end{equation*}
is a continuous isomorphism.

Corollary~\ref{equivalent norm $E_T$} and Lemma~\ref{equivnorm2} imply that for every positive $T$ the spaces $\mathbb{E}_T$ and $\mathbb{F}_T$ endowed with the norms
\begin{equation*}
\vertiii{\gamma}_{\boldsymbol{E}_T}:=\vertiii{\gamma}_{W_p^{1,2}\left((0,T)\times(0,1);(\mathbb{R}^n)^3\right)}=\left\lVert\gamma\right\rVert_{W_p^{1,2}\left((0,T)\times(0,1);(\mathbb{R}^n)^3\right)}+\left\lVert\gamma(0)\right\rVert_{W_p^{2-\nicefrac{2}{p}}\left((0,1);(\mathbb{R}^n)^3\right)}
\end{equation*}
and
\begin{align*}
\vertiii{(f,\eta,0,b,\psi)}_{\mathbb{F}_T}:=&\left\lVert f\right\rVert_{L_p\left((0,T);L_p((0,1);(\mathbb{R}^n)^3)\right)}+\vertiii{\eta}_{W_p^{1-\nicefrac{1}{2p}}\left((0,T);(\mathbb{R}^n)^3\right)}\\
&+\vertiii{b}_{W_p^{\nicefrac{1}{2}-\nicefrac{1}{2p}}\left((0,T);\mathbb{R}^n\right)}+\left\lVert\psi\right\rVert_{W_p^{2-\nicefrac{2}{p}}\left((0,1);(\mathbb{R}^n)^3\right)}\,,
\end{align*}
respectively, are Banach spaces. Given a linear operator $A:\mathbb{F}_T\to\mathbb{E}_T$ we let
\begin{equation*}
\vertiii{A}_{\mathcal{L}\left(\mathbb{F}_T,\mathbb{E}_T\right)}:=\sup\{\vertiii{A(f,\eta,0,b,\psi)}_{\boldsymbol{E}_T}:(f,\eta,0,b,\psi)\in\mathbb{F}_T,\vertiii{(f,\eta,0,b,\psi)}_{\mathbb{F}_T}\leq 1\}\,.
\end{equation*}

\begin{lem}\label{Luniformlybounded}
	Let $p\in (3,\infty)$. For all $T_0>0$ there exists a constant $c(T_0,p)$ such that
	\begin{equation*}
	\sup_{T\in (0,T_0]}\vertiii{L_T^{-1}}_{\mathcal{L}(\mathbb{F}_T,\mathbb{E}_T)}\leq c(T_0,p)\,.
	\end{equation*}
\end{lem}
\begin{proof}
	Let $T\in (0,T_0]$ be arbitrary, $\left(f,\eta,0,b,\psi\right)\in\mathbb{F}_T$ and $E_{T_0}b:=\left(Eb\right)_{|(0,T_0)}$, $E_{T_0}\eta:=\left(E\eta\right)_{|(0,T_0)}$ where $E$ is the extension operator defined in Lemma~\ref{extensionboundarydata}. Extending $f$ by $0$ to $E_{T_0}f\in L_p\left((0,T_0);L_p\left((0,1);(\mathbb{R}^n)^3\right)\right)$ we observe that $\left(E_{T_0}f,E_{T_0}\eta,0,E_{T_0}b,\psi\right)$ lies in $\mathbb{F}_{T_0}$. As $L_T$ and $L_{T_0}$ are isomorphisms, there exist unique $\gamma\in\mathbb{E}_T$ and $\widetilde{\gamma}\in\mathbb{E}_{T_0}$ such that $L_T\gamma=(f,\eta,0,b,\psi)$ and $L_{T_0}\widetilde{\gamma}=\left(E_{T_0}f,E_{T_0}\eta,0,E_{T_0}b,\psi\right)$ satisfying
	\begin{equation*}
	L_T\gamma=(f,\eta,0,b,\psi)=\left(E_{T_0}f,E_{T_0}\eta,0,E_{T_0}b,\psi\right)_{|(0,T)}=\left(L_{T_0}\widetilde{\gamma}\right)_{|(0,T)}=L_T\left(\widetilde{\gamma}_{|(0,T)}\right)
	\end{equation*}
	and thus $\gamma=\widetilde{\gamma}_{|(0,T)}$. Using Theorem~\ref{exlin}, Lemma~\ref{extensionboundarydata} and the equivalence of norms on $\boldsymbol{E}_{T_0}$ this implies
	\begin{align*}
	&\vertiii{L_T^{-1}\left(f,\eta,0,b,\psi\right)}_{\boldsymbol{E}_T}=\vertiii{\left(L_{T_0}^{-1}\left(E_{T_0}f,E_{T_0}\eta,0,E_{T_0}b,\psi\right)\right)_{|(0,T)}}_{\boldsymbol{E}_T}\\
	&\leq \vertiii{L_{T_0}^{-1}\left(E_{T_0}f,E_{T_0}\eta,0,E_{T_0}b,\psi\right)}_{\boldsymbol{E}_{T_0}}\leq c\left(T_0,p\right)\left\lVert L_{T_0}^{-1}\left(E_{T_0}f,E_{T_0}\eta,0,E_{T_0}b,\psi\right)\right\rVert_{\boldsymbol{E}_{T_0}}\\
	&\leq c\left(T_0,p\right)\left\lVert\left(E_{T_0}f,E_{T_0}\eta,0,E_{T_0}b,\psi\right)\right\rVert_{\mathbb{F}_{T_0}}\leq c\left(T_0,p\right)\vertiii{(f,\eta,0,b,\psi)}_{\mathbb{F}_T}\,.
	\end{align*}
\end{proof}

\subsection{Existence and uniqueness of the Special Flow}\label{sectionspecialflow}


Given $M$ positive we introduce the notation
\begin{equation*}
\overline{B_M}:=\left\{\gamma\in\boldsymbol{E}_T:\vertiii{\gamma}_{\boldsymbol{E}_T}\leq M\right\}\,.
\end{equation*}

This section is devoted to the proof of the following:

\begin{thm}\label{existenceanalyticprob}
	Let $p\in(3,\infty)$ and 
	let $\sigma=(\sigma^1,\sigma^2,\sigma^3)$ be an admissible initial parametrisation. 
	There exists a positive radius $M$ and a positive time $T$ such that the system~\eqref{problema}
	has a unique solution $\mathcal{E}\sigma$ in 
	$
	\boldsymbol{E}_T\cap\overline{B_M}\,.
	$
\end{thm}

Given an admissible initial parametrisation $\sigma$ and $T>0$ we consider the complete metric spaces
\begin{align*}
\mathbb{E}^\sigma_T&:=\{\gamma\in\mathbb{E}_T\;\text{such that}\;\gamma_{\vert t=0}=\sigma\,\text{ and }\gamma_{|x=1}=\sigma(1)\}\,,\\
\mathbb{F}^\sigma_T&:=\mathbb{F}_T\cap\left( L_p\left((0,T);L_p\left((0,1);(\mathbb{R}^n)^3\right)\right)\times\{\sigma(1)\}\times\{0\}\times W_p^{\nicefrac{1}{2}-\nicefrac{1}{2p}}\left((0,T);\mathbb{R}^n\right)\times\{\sigma\}\right)\,.
\end{align*}

\begin{lem}\label{trace}
	Let $p\in(3,\infty)$, $T>0$ and $\sigma=(\sigma^1,\sigma^2,\sigma^3)$ be an admissible initial parametrisation. Then the space $\mathbb{E}_T^\sigma$ is non-empty.
\end{lem}
\begin{proof}
As $\sigma$ is an admissible initial parametrisation, one easily checks that $f\equiv 0$, $\eta\equiv\sigma(1)$, $b\equiv 0$ and $\psi\equiv \sigma$ is an admissible right hand side for system~\eqref{linsys}. In other words, $\left(0,\sigma(1),0,0,\sigma\right)\in\mathbb{F}_T$ and hence Theorem~\ref{exlin} yields the existence of $\varrho\in\mathbb{E}_T$ with $L_T\varrho=\left(0,\sigma(1),0,0,\sigma\right)$. In particular, $\varrho_{|t=0}=\sigma$ and $\varrho_{|x=1}=\sigma(1)$ which gives $\varrho\in\mathbb{E}_T^\sigma$. 
\end{proof}
\begin{lem}\label{curveregular}
	Let $p\in(3,\infty)$ and
	\begin{equation*}
	\boldsymbol{c}:=\frac{1}{2}\min_{i\in\{1,2,3\},x\in[0,1]}\vert\sigma^i_x(x)\vert\,.
	\end{equation*}
Given $T_0>0$ and $M>0$ there exists a time $\widetilde{T}(\boldsymbol{c},M)\in (0,T_0]$ such that for all $\gamma\in \mathbb{E}_{T}^\sigma\cap\overline{B_M}$ with $T\in[0,\widetilde{T}(\boldsymbol{c},M)]$ it holds
\begin{equation*}
\inf_{x\in[0,1],t\in[0,T],i\in\{1,2,3\}}\left\lvert \gamma^i_x(t,x)\right\rvert \geq \boldsymbol{c}\,.
\end{equation*}
In particular, the curves $\gamma^i(t)$ are regular for all $t\in [0,T]$.
\end{lem}
\begin{proof}
Let $p\in(3,\infty)$, $\theta\in\left(\frac{1+\nicefrac{1}{p}}{2-\nicefrac{2}{p}},1\right)$ and $\delta\in\left(0,1-\nicefrac{1}{p}\right)$. By Theorem~\ref{uniformcalphac1} there exists a constant $C(T_0,p,\theta,\delta)>0$ such that for all $T\in (0,T_0]$ and all $\gamma\in\mathbb{E}_T^\sigma\cap\overline{B_M}$ with $\alpha:=(1-\theta)(1-\nicefrac{1}{p}-\delta)$ it holds
\begin{equation*}
\lVert\gamma\rVert_{C^\alpha\left([0,T];C^1([0,1];(\mathbb{R}^n)^3)\right)}\leq C(T_0,p,\theta,\delta)\vertiii{\gamma}_{\boldsymbol{E}_{T}}\leq C(T_0,p,\theta,\delta)M\,,
\end{equation*}
which implies in particular for all $t\in [0,T]$,
\begin{equation*}
\lVert \gamma(t)-\sigma\rVert_{C^1([0,1];(\mathbb{R}^n)^3)}\leq T^\alpha C(T_0,p,\theta,\delta)M\,.
\end{equation*}
We let $\widetilde{T}(\boldsymbol{c},M)$ be so small that $\widetilde{T}(\boldsymbol{c},M)^\alpha C(T_0,p,\theta,\delta)M\leq\boldsymbol{c}$. Then it follows for all $\gamma\in \mathbb{E}_T^\sigma$ with $T\in (0,\widetilde{T}(\boldsymbol{c},M))$,
\begin{equation*}
\inf_{t\in[0,T],x\in[0,1]}\vert\gamma^i_x(t,x)\vert\geq \inf_{x\in[0,1]}\vert\sigma^i_x(x)\vert -\sup_{t\in[0,T],x\in[0,1]}\vert\gamma^i_x(t,x)-\gamma^i_x(0,x)\vert \geq \boldsymbol{c}\,.
\end{equation*}
\end{proof}

Let us now define the operator $N_T$ that encodes the non--linearity of our problem.
The map $N_T:\mathbb{E}^\sigma_T \to \mathbb{F}_T^\sigma$ is given by 
$\gamma \mapsto \left(N_T^1(\gamma),\gamma_{|x=1},0,N_T^2(\gamma),\gamma_{|t=0}\right)$ where
the two components $N^1_T,N^2_T$ are defined as
\begin{align*}
N^{1}_T:&
\begin{cases}
\mathbb{E}^\sigma_T &\to L_p((0,T);L_p((0,1);(\mathbb{R}^n)^3))\,,\\
\gamma &\mapsto
f(\gamma)\,,
\end{cases}\\
N^2_{T}:&
\begin{cases}
\mathbb{E}^\sigma_T&\to   W_p^{\nicefrac{1}{2}-\nicefrac{1}{2p}}((0,T);\mathbb{R}^n)\,,\\
\gamma&\mapsto b(\gamma)
\end{cases}
\end{align*}
with 
\begin{align*}
f(\gamma)^i(t,x)&:=\left(\frac{1}{\left|\gamma^i_{x}(t,x)\right|^{2}}-\frac{1}{\left|\sigma^i_x(x)\right|^{2}}\right)
\gamma^i_{xx}(t,x)\,,\\
b(\gamma)(t)&:=\sum_{i=1}^3 \left(\left(\frac{1}{\vert \gamma^i_x(t,0)\vert}
-\frac{1}{\vert\sigma^i_x(0)\vert}\right)\gamma^i_x(t,0) +
\frac{\sigma^i_x(0)\left\langle \gamma^i_x(t,0),\sigma^i_x(0)\right\rangle}{\vert \sigma^i_x(0)\vert^3}
\right)
\end{align*}
defined by the right hand side of~\eqref{deff} and~\eqref{defpsi}, respectively.

\begin{prop}
Let $p\in(3,\infty)$ and $M$ be positive. Then for all $T\in (0,\widetilde{T}(\boldsymbol{c},M)]$ the map 
\begin{equation*}
N_T:\mathbb{E}_T^\sigma\cap \overline{B_M}\to \mathbb{F}_T^\sigma\,,\quad N_T(\gamma):= \left(N_T^1(\gamma),\gamma_{|x=1},0,N_T^2(\gamma),\gamma_{|t=0}\right)
\end{equation*}
is well-defined.
\end{prop}
\begin{proof}
Let $T\in (0,\widetilde{T}(\boldsymbol{c},M)]$ and $\gamma\in\mathbb{E}_T^\sigma\cap \overline{B_M}$ be given.
Lemma~\ref{curveregular} implies
\begin{align*}
&\Big\Vert \left(\frac{1}{\vert \gamma^i_x\vert^2}-\frac{1}{\vert\sigma^i_x\vert^2}\right) 
\gamma^i_{xx}\Big\Vert_{L_p((0,T);L_p((0,1);\mathbb{R}^n))}
=\int_0^T\int_0^1
\left\lvert \frac{1}{\vert \gamma^i_x\vert^2}-\frac{1}{\vert\sigma^i_x\vert^2}\right\rvert^p
\vert \gamma^i_{xx} \vert^p\,\mathrm{d}x\,\mathrm{d}t\\
&\leq C\left(\sup_{x\in [0,1], t\in [0,T]}\frac{1}{\vert \gamma^i_x\vert^{2p}}+
\sup_{x\in [0,1]}\frac{1}{\vert \sigma^i_x\vert^{2p}}\right)
\int_0^T\int_0^1\vert \gamma^i_{xx} \vert^p\,\mathrm{d}x\,\mathrm{d}t\\
&\leq C(\boldsymbol{c})\Vert \gamma^i_{xx}\Vert^p_{L_p((0,T);L_p((0,1);\mathbb{R}^n))}\leq C(\boldsymbol{c},M)<\infty\,.
\end{align*}

We now show that $N_T^2(\gamma)$ lies in $W_p^{\nicefrac{1}{2}-\nicefrac{1}{2p}}\left((0,T);\mathbb{R}^n\right)$. 
Let $h:\mathbb{R}^n\to\mathbb{R}^n$ be a smooth function such that $h(p)=\frac{p}{\vert p\vert}$ for all $p\in \mathbb{R}^n\setminus B_{\nicefrac{\boldsymbol{c}}{2}}(0)$. Then one observes that for all $t\in [0,T]$ 
\begin{equation}\label{identityb}
b(\gamma)(t)=\sum_{i=1}^3 h(\gamma^i_x(t))-\left(Dh\right)(\sigma^i_x)\gamma^i_x(t)
\end{equation}
where we omitted the evaluation in $x=0$ to ease notation. Each term in the sum can be expressed as
\begin{align*}
h(\gamma^i_x(t))-\left(Dh\right)(\sigma^i_x)\gamma^i_x(t)&=\int_0^1(Dh)(\xi\gamma^i_x(t)+(1-\xi)\sigma^i_x)\mathrm{d}\xi\,(\gamma^i_x(t)-\sigma^i_x)\\
&\phantom{=}-(Dh)(\sigma^i_x)\left(\gamma^i_x(t)-\sigma^i_x\right)+h(\sigma^i_x)-Dh(\sigma^i_x)\sigma^i_x\,.
\end{align*}
All terms that are constant in $t$ are smooth in $t$ and by Lemma~\ref{derivativesol} we have
\begin{equation*}
t\mapsto \gamma^i_x(t,0) \in W_p^{\nicefrac{1}{2}-\nicefrac{1}{2p}}\left((0,T);\mathbb{R}^n\right)\,.
\end{equation*} 
As $W_p^{\nicefrac{1}{2}-\nicefrac{1}{2p}}\left((0,T);\mathbb{R}\right)$ is a Banach algebra according to Proposition~\ref{Banachalgebra}, it only remains to show  
\begin{equation*}
t\mapsto \int_0^1(Dh)(\xi\gamma^i_x(t,0)+(1-\xi)\sigma^i_x(0))\mathrm{d}\xi \in W_p^{\nicefrac{1}{2}-\nicefrac{1}{2p}}\left((0,T);\mathbb{R}^{n\times n}\right)
\end{equation*}
which follows from the second assertion in Proposition~\ref{Banachalgebra}. Observe that $\gamma_{|x=1}=\sigma(1)$ and $\gamma_{|t=0}=\sigma$ by definition of $\mathbb{E}_T^\sigma$. As
\begin{equation*}
N_T^2(\gamma)_{|t=0}=\sum_{i=1}^3\frac{\sigma^i_x(0)}{\vert\sigma^i_x(0)\vert}=0=-\sum_{i=1}^{3}\left(\frac{\sigma^i_x(0)}{\vert\sigma^i_x(0)\vert}-\frac{\sigma^i_x(0)\left\langle\sigma^i_x(0),\sigma^i_x(0)\right\rangle}{\vert\sigma^i_x(0)\vert^3}\right)
\end{equation*}
and as $\sigma^i(0)=\sigma^j(0)$, $\sigma^i(1)=\gamma^i(0,1)$, we may conclude that 
\begin{equation*}
\left(N_T^1(\gamma),\gamma_{|x=1},0,N_T^2(\gamma),\gamma_{|t=0}\right)=\left(N_T^1(\gamma),\sigma(1),0,N_T^2(\gamma),\sigma\right)\in\mathbb{F}_T^\sigma.
\end{equation*}
\end{proof}
\begin{cor}
	Let $p\in(3,\infty)$ and $M$ be positive. Then for all $T\in (0,\widetilde{T}(\boldsymbol{c},M)]$ the map 
	\begin{equation*}
	K_T:\mathbb{E}_T^\sigma\cap \overline{B_M}\to \mathbb{E}_T^\sigma\,,\quad 
	K_T:=L_T^{-1}N_T
	\end{equation*}
	is well-defined.
\end{cor}
\begin{proof}
	Let $T\in(0,\widetilde{T}(\boldsymbol{c},M)]$ and $\gamma\in\mathbb{E}_T^\sigma\cap \overline{B_M}$. By the previous proof we have
	\begin{equation*}
	N_T(\gamma)=\left(N_T^1(\gamma),\gamma_{|x=1},0,N_T^2(\gamma),\gamma_{|t=0}\right)\in \mathbb{F}_T^\sigma\subset\mathbb{F}_T
	\end{equation*}
	and thus in particular
	\begin{equation*}
	K_T(\gamma)=L_T^{-1}(N_T(\gamma))\in\mathbb{E}_T\,.
	\end{equation*}
	To verify that $K_T(\gamma)$ lies in $\mathbb{E}_T^\sigma$ we observe that 
	\begin{align*}
	K_T(\gamma)_{|t=0}&=\left(L_T\left(K_T(\gamma)\right)\right)_5=N_T(\gamma)_5=\gamma_{|t=0}=\sigma\,,\\
	K_T(\gamma)_{|x=1}&=\left(L_T\left(K_T(\gamma)\right)\right)_2=N_T(\gamma)_2=\gamma_{|x=1}=\sigma(1)\,.
	\end{align*}
\end{proof}
\begin{prop}\label{Kcontraction}
Let $p\in(3,\infty)$ and $M$ be positive. There exists $T(\boldsymbol{c},M)\in (0,\widetilde{T}(\boldsymbol{c},M)]$ such that for all $T\in (0,T(\boldsymbol{c},M)]$
the map 
$K_T:\mathbb{E}^\sigma_T\cap \overline{B_M}\to\mathbb{E}^\sigma_T$ 
is a contraction. 
\end{prop}

\begin{proof}
Let $T\in(0,\widetilde{T}(\boldsymbol{c},M)]$ and $\gamma,\widetilde{\gamma}\in\mathbb{E}^\sigma_T\cap\overline{B_M}$ be fixed.
We begin by estimating
\begin{equation*}
\Vert N^1_T(\gamma)-N^1_T(\widetilde{\gamma}) \Vert_{L_p\left((0,T);L_p\left((0,1);(\mathbb{R}^n)^3\right)\right)}=
\Vert f(\gamma)-f(\widetilde{\gamma})\Vert_{L_p\left((0,T);L_p\left((0,1);(\mathbb{R}^n)^3\right)\right)}\,.
\end{equation*}
The $i$-th component of $f(\gamma)-f(\widetilde{\gamma})$ is given by
\begin{align*}
&\left(\frac{1}{\vert\gamma_x^i\vert^2}-\frac{1}{\vert\sigma_x^i\vert^2}\right)
(\gamma_{xx}^i-\widetilde{\gamma}_{xx}^i)
+\left(\frac{1}{\vert\gamma_x^i\vert^2}-\frac{1}{\vert\widetilde{\gamma}_x^i\vert^2}\right)\widetilde{\gamma}^i_{xx}\\
&=\left(\frac{1}{\vert \gamma^i_x\vert^2\vert \sigma^i_x\vert}
+\frac{1}{\vert \gamma^i_x\vert\vert \sigma^i_x\vert^2}\right)
\left(\vert \sigma^i_x\vert-\vert \gamma^i_x\vert\right)
(\gamma_{xx}^i-\widetilde{\gamma}_{xx}^i)\\
&\quad +\left(\frac{1}{\vert \gamma^i_x\vert^2 \vert \widetilde{\gamma}^i_x\vert}
+\frac{1}{\vert \gamma^i_x\vert\vert \widetilde{\gamma}^i_x\vert^2}\right)
\left(\vert \widetilde{\gamma}^i_x\vert-\vert \gamma^i_x\vert\right)\widetilde{\gamma}^i_{xx}\,.
\end{align*}
Lemma~\ref{curveregular} implies
\begin{align*}
\sup_{t\in[0,T],x\in[0,1]}\left\vert\frac{1}{\vert \gamma^i_x\vert^2\vert \sigma^i_x\vert}
+\frac{1}{\vert \gamma^i_x\vert\vert \sigma^i_x\vert^2}\right\vert\leq C(\boldsymbol{c})<\infty\,,
\end{align*}
and
\begin{align*}
\sup_{t\in[0,T],x\in[0,1]}\left\lvert
\frac{1}{\vert \gamma^i_x\vert^2\vert \widetilde{\gamma}^i_x\vert}
+\frac{1}{\vert \gamma^i_x\vert\vert \widetilde{\gamma}^i_x\vert^2}
\right\rvert\leq C(\boldsymbol{c})<\infty\,.
\end{align*}
Hence we obtain
\begin{align*}
&\left\lVert f(\gamma)^i-f(\widetilde{\gamma})^i\right\rVert_{L_p(0,T;L_p\left((0,1);(\mathbb{R}^n)^3\right))}\\
&\leq C(\boldsymbol{c})\left(\left\lVert\left(\vert \sigma^i_x\vert-\vert \gamma^i_x\vert\right)
(\gamma_{xx}^i-\widetilde{\gamma}_{xx}^i)\right\rVert_{L_p\left((0,T);L_p(0,1);\mathbb{R}^n\right)}+\left\lVert\left(\vert \widetilde{\gamma}^i_x\vert-\vert \gamma^i_x\vert\right)\widetilde{\gamma}^i_{xx}\right\rVert_{L_p\left((0,T);L_p(0,1);\mathbb{R}^n\right)}\right)\\
&\leq C(\boldsymbol{c})\left(\sup_{t\in[0,T],x\in[0,1]}\left\vert \vert \sigma^i_x(x)\vert-\vert \gamma^i_x(t,x)\vert\right\vert\left\lVert\gamma^i_{xx}-\widetilde{\gamma}^i_{xx}\right\rVert_{L_p\left((0,T);L_p(0,1);\mathbb{R}^n)\right)}\right.\\
&\left. \phantom {C(\boldsymbol{c})} \,\,\,\,\,\,\,\,+\sup_{t\in[0,T],x\in[0,1]}\left\vert \vert \widetilde{\gamma}^i_x(t,x)\vert-\vert \gamma^i_x(t,x)\vert\right\vert \left\lVert\widetilde{\gamma}^i_{xx}\right\rVert_{L_p\left((0,T);L_p((0,1);\mathbb{R}^n)\right)}\right)\\
&\leq C(\boldsymbol{c})\sup_{t\in[0,T],x\in[0,1]}\left\lvert\sigma^i_x(x)-\gamma^i_x(t,x)\right\rvert\vertiii{\gamma-\widetilde{\gamma}}_{\boldsymbol{E}_T}\\
&\quad + C(\boldsymbol{c})\sup_{t\in[0,T],x\in[0,1]}\left\lvert\widetilde{\gamma}^i_x(t,x)-\gamma^i_x(t,x)\right\rvert\vertiii{\widetilde{\gamma}}_{\boldsymbol{E}_T}\,.
\end{align*}
Let $\theta\in\left(\frac{1+\nicefrac{1}{p}}{2-\nicefrac{2}{p}},1\right)$, $\delta\in\left(0,1-\nicefrac{1}{p}\right)$ be fixed and define $\alpha:=(1-\theta)(1-\nicefrac{1}{p}-\delta)$. Theorem~\ref{uniformcalphac1} implies
\begin{align*}
&\sup_{t\in[0,T],x\in[0,1]}\left\lvert\sigma^i_x(x)-\gamma^i_x(t,x)\right\rvert=\sup_{t\in[0,T]}\left\lVert\gamma^i_x(0)-\gamma^i_x(t)\right\rVert_{C([0,1];\mathbb{R}^n)}\\
&\leq \sup_{t\in[0,T]}\left\lVert\gamma^i(t)-\gamma^i(0)\right\rVert_{C^1([0,1];\mathbb{R}^n)}\leq \sup_{t\in[0,T]}t^\alpha\left\lVert\gamma^i\right\rVert_{C^\alpha\left([0,T];C^1([0,1];\mathbb{R}^n)\right)}\\
&\leq T^\alpha\left\lVert\gamma^i\right\rVert_{C^\alpha\left([0,T];C^1([0,1];\mathbb{R}^n)\right)}\leq T^\alpha C(T_0,p,\theta,\delta)\vertiii{\gamma}_{\boldsymbol{E}_T}\leq C(M)T^\alpha\,.
\end{align*}
Similarly we obtain
\begin{align*}
\sup_{t\in[0,T],x\in[0,1]}\left\lvert\widetilde{\gamma}^i_x(t,x)-\gamma^i_x(t,x)\right\rvert&= \sup_{t\in[0,T],x\in[0,1]}\left\lvert\left(\widetilde{\gamma}^i_x-\gamma^i_x\right)(t,x)-\left(\widetilde{\gamma}^i_x-\gamma^i_x\right)(0,x)\right\rvert\\
&\leq\sup_{t\in[0,T]}\left\lVert\left(\widetilde{\gamma}^i-\gamma^i\right)(t)-\left(\widetilde{\gamma}^i-\gamma^i\right)(0)\right\rVert_{C^1([0,1];\mathbb{R}^n)}\\
&\leq \sup_{t\in[0,T]} t^\alpha\left\lVert\widetilde{\gamma}^i-\gamma^i\right\rVert_{C^\alpha\left([0,T];C^1([0,1];\mathbb{R}^n)\right)}\leq CT^\alpha\vertiii{\widetilde{\gamma}-\gamma}_{\boldsymbol{E}_T}\,.
\end{align*}
This allows us to conclude
\begin{equation*}
\left\lVert f(\gamma)-f(\widetilde{\gamma})\right\rVert_{L_p\left((0,T);L_p((0,1);(\mathbb{R}^n)^3)\right)}\leq C(\boldsymbol{c},M)T^\alpha\vertiii{\gamma-\widetilde{\gamma}}_{\boldsymbol{E}_T}\,.
\end{equation*}
We proceed by estimating 

\begin{equation*}
\left\lVert N_T^2(\gamma)-N_T^2(\widetilde{\gamma})\right\rVert_{W_p^{\nicefrac{1}{2}-\nicefrac{1}{2p}}((0,T);\mathbb{R}^n)}=\left\lVert b(\gamma)-b(\widetilde{\gamma})\right\rVert_{W_p^{\nicefrac{1}{2}-\nicefrac{1}{2p}}((0,T);\mathbb{R}^n)}\,.
\end{equation*}
Let $T\in(0,\widetilde{T}(\boldsymbol{c},M)]$ be fixed and $h:\mathbb{R}^n\to\mathbb{R}^n$ be a smooth function such that $h(p)=\frac{p}{\vert p\vert}$ on $\mathbb{R}^n\setminus B_{\nicefrac{\boldsymbol{c}}{2}}(0)$. As for all $t\in[0,T]$ and $\eta\in\mathbb{E}_T^\sigma\cap\overline{B_M}$,
\begin{equation*}
\vert\eta^i_x(t,0)\vert \geq \boldsymbol{c}\,,
\end{equation*}
we may conclude that for all $\gamma,\widetilde{\gamma}\in\mathbb{E}_T^\sigma\cap\overline{B_M}$, the function
\begin{equation*}
t\mapsto g^i(t):=\int_0 ^1\left(Dh\right)(\xi\gamma^i_x(t,0)+(1-\xi)\widetilde{\gamma}^i_x(t,0))\mathrm{d}\xi
\end{equation*}
lies in $W_p^{\nicefrac{1}{2}-\nicefrac{1}{2p}}(0,T;\mathbb{R}^{n\times n})$. To ease notation we let $s:=\nicefrac{1}{2}-\nicefrac{1}{2p}$. Observe that $g^i(0)=(Dh)(\sigma^i_x(0))$ and thus using identity~\eqref{identityb} we obtain
\begin{equation*}
b(\gamma)(t)-b(\widetilde{\gamma})(t)=\sum_{i=1}^3 \left(g^i(t)-g^i(0)\right)\left(\gamma^i_x(t,0)-\widetilde{\gamma}^i_x(t,0)\right)\,.
\end{equation*}
Using the product estimate in Proposition~\ref{Banachalgebra} we obtain
\begin{align*}
\left\lVert b(\gamma)-b(\widetilde{\gamma})\right\rVert_{W_p^s\left((0,T);\mathbb{R}^n\right)}&\leq\sum_{i=1}^3\left\lVert\left(g^i-g^i(0)\right)\left(\gamma^i_x(\cdot,0)-\widetilde{\gamma}^i_x(\cdot,0)\right)\right\rVert_{W_p^s\left((0,T);\mathbb{R}^n\right)}\\
&\leq\sum_{i=1}^3 \left\lVert g^i-g^i(0)\right\rVert_{ C([0,T];\mathbb{R}^{n\times n})}\left\lVert\gamma^i_x(\cdot,0)-\widetilde{\gamma}^i_x(\cdot,0)\right\rVert_{W_p^s\left(0,T;\mathbb{R}^n\right)}\\
&\; \; \; \; \quad +\left\lVert g^i-g^i(0)\right\rVert_{ W_p^s\left(0,T;\mathbb{R}^{n\times n}\right)}\left\lVert\gamma^i_x(\cdot,0)-\widetilde{\gamma}^i_x(\cdot,0)\right\rVert_{C([0,T];\mathbb{R}^n)}\,.
\end{align*}
As $s-\frac{1}{p}>0$ due to $p\in(3,\infty)$ there exists $\beta\in(0,1)$ such that 
\begin{align*}
W_p^s\left(0,T;\mathbb{R}^n\right)\hookrightarrow C^\beta\left([0,T];\mathbb{R}^n\right)
\end{align*}
with embedding constant $C(s,p)$. This implies in particular
\begin{equation*}
\sup_{t\in[0,T]}\vert g^i(t)-g^i(0)\vert\leq T^\beta\left\lVert g^i\right\rVert_{C^\beta\left([0,T];\mathbb{R}^{n\times n}\right)}\leq T^\beta C(s,p)\left\lVert g^i\right\rVert_{W_p^s\left((0,T);\mathbb{R}^{n\times n}\right)}\,.
\end{equation*}
Reading carefully through the estimates in Proposition~\ref{Banachalgebra} we observe that
\begin{equation*}
\left\lVert g^i\right\rVert_{W_p^s\left((0,T);\mathbb{R}^{n\times n}\right)}\leq C(T_0,\boldsymbol{c},M)\,.
\end{equation*}
Furthermore, given $\theta\in\left(\frac{1+\nicefrac{1}{p}}{2-\nicefrac{2}{p}},1\right)$ and $\delta\in\left(0,1-\nicefrac{1}{p}\right)$, Theorem~\ref{uniformcalphac1} implies with $\alpha:=(1-\theta)\left(1-\nicefrac{1}{p}-\delta\right)>0$ the estimate
\begin{align*}
\sup_{t\in[0,T]}\left\lvert\gamma^i_x(t,0)-\widetilde{\gamma}^i_x(t,0)\right\rvert&=\sup_{t\in[0,T]}\left\lvert\left(\gamma^i_x-\widetilde{\gamma}^i_x\right)(t,0)-(\gamma^i_x-\widetilde{\gamma}^i_x)(0,0)\right\rvert\\
&\leq \sup_{t\in[0,T]}\left\lVert\left(\gamma^i-\widetilde{\gamma}^i\right)(t)-\left(\gamma^i-\widetilde{\gamma}^i\right)(0)\right\rVert_{C^1\left([0,1];\mathbb{R}^n\right)}\\
&\leq T^\alpha\left\lVert\gamma^i-\widetilde{\gamma}^i\right\rVert_{C^{\alpha}\left([0,T];C^1\left([0,1],\mathbb{R}^n\right)\right)}\leq T^\alpha\vertiii{\gamma-\widetilde{\gamma}}_{\boldsymbol{E}_T}\,.
\end{align*}
This allows us to conclude
\begin{align*}
\left\lVert b(\gamma)-b(\widetilde{\gamma})\right\rVert_{W_p^{\nicefrac{1}{2}-\nicefrac{1}{2p}}\left(0,T;\mathbb{R}^n\right)}\leq C(s,p,T_0,\boldsymbol{c},M)T^\alpha\vertiii{\gamma-\widetilde{\gamma}}_{\boldsymbol{E}_T}\,.
\end{align*}
Finally, Lemma~\ref{Luniformlybounded} implies for all $T\in(0,\widetilde{T}(\boldsymbol{c},M)]$,
\begin{align*}
&\vertiii{K_T(\gamma)-K_T(\widetilde{\gamma})}_{\boldsymbol{E}_T}=\vertiii{L_T^{-1}\left(N_T(\gamma)-N_T(\widetilde{\gamma})\right)}_{\boldsymbol{E}_T}\leq c(T_0,p)\vertiii{N_T(\gamma)-N_T(\widetilde{\gamma})}_{\mathbb{F}_T}\\
&=c(T_0,p)\left(\Vert f(\gamma)-f(\widetilde{\gamma})\Vert_{L_p\left((0,T);L_p\left((0,1);(\mathbb{R}^n)^3\right)\right)}+\left\lVert b(\gamma)-b(\widetilde{\gamma})\right\rVert_{W_p^{\nicefrac{1}{2}-\nicefrac{1}{2p}}\left(0,T;\mathbb{R}^n\right)}\right)\\
&\leq C(T_0,p,\boldsymbol{c},M)T^{\min\{\alpha,\beta\}}\vertiii{\gamma-\widetilde{\gamma}}_{\boldsymbol{E}_T}\,.
\end{align*}
This completes the proof.
\end{proof}

To conclude the existence of a solution with the Banach Fixed Point Theorem we have to show that there exists a radius $M>0$ such that $K_T$ is a self-mapping of $\mathbb{E}_T^\sigma\cap\overline{B_M}$.

\begin{prop}\label{selfmapping}
	Let $p\in(3,\infty)$. There exists a positive radius $M$ depending on $\boldsymbol{c}$ and the norm of $\sigma$ in $W_p^{2-\nicefrac{2}{p}}\left((0,1);(\mathbb{R}^n)^3\right)$ and a positive time $\widehat{T}(\boldsymbol{c},M)$ such that for all $T\in (0,\widehat{T}(\boldsymbol{c},M)]$ the map
	\begin{equation*}
	K_T:\mathbb{E}_T^\sigma\cap\overline{B_M}\to\mathbb{E}_T^\sigma\cap\overline{B_M}
	\end{equation*}
	is well-defined.
\end{prop}
\begin{proof}
	We let $T_0= 1$ and define
	\begin{equation*}
	M:=2\max\left\{\sup_{T\in(0,1]}\vertiii{L_T^{-1}}_{\mathcal{L}\left(\mathbb{F}_T,\mathbb{E}_T\right)},1\right\}\max\left\{\vertiii{\mathcal{L}\sigma}_{\boldsymbol{E}_1},\vertiii{\left(N_{1}^1(\mathcal{L}\sigma),\sigma(1),0,N_{1}^2(\mathcal{L}\sigma),\sigma\right)}_{\mathbb{F}_1}\right\}
	\end{equation*}	
	where $\mathcal{L}\sigma:=L_1^{-1}\left(0,\sigma(1),0,0,\sigma\right)$ denotes the extension defined in Lemma~\ref{trace} with $T=1$.
	In particular, $\mathcal{L}\sigma$ lies in $\mathbb{E}_T^\sigma\cap\overline{B_M}$ for all $T\in (0,1]$. Moreover, for all $T\in (0,1]$ we have 
	\begin{equation*}
	\vertiii{K_T\left(\mathcal{L}\sigma\right)}_{\boldsymbol{E}_T}\leq \sup_{T\in(0,1]}\vertiii{L_T^{-1}}_{\mathcal{L}\left(\mathbb{F}_T,\mathbb{E}_T\right)}\vertiii{\left(N_{1}^1(\mathcal{L}\sigma),\sigma(1),0,N_{1}^2(\mathcal{L}\sigma),\sigma\right)}_{\mathbb{F}_T}\leq \nicefrac{M}{2}\,.
	\end{equation*}
	Let $T\left(\boldsymbol{c},M\right)$ be the time as in Proposition~\ref{Kcontraction}. Given $T\in (0,T\left(\boldsymbol{c},M\right)]$ and $\gamma\in \mathbb{E}_T^\sigma\cap\overline{B_M}$ we observe that for some $\beta\in(0,1)$,
	\begin{equation*}
	\vertiii{ K_T\left(\gamma\right)-K_T\left(\mathcal{L}\sigma\right)}_{\boldsymbol{E}_T}\leq C\left(\boldsymbol{c},M\right)T^\beta\vertiii{\gamma-\mathcal{L}\sigma}_{\boldsymbol{E}_T}\leq C\left(\boldsymbol{c},M\right)T^\beta 2M\,.
	\end{equation*}
	We choose a time $\widehat{T}(\boldsymbol{c},M)\in (0,T\left(\boldsymbol{c},M\right)]$ so small that for all $T\in (0,\widehat{T}(\boldsymbol{c},M)]$ it holds $C\left(\boldsymbol{c},M\right)T^\beta 2M\leq \nicefrac{M}{2}$. Finally, we conclude for all $T\in (0,\widehat{T}(\boldsymbol{c},M)]$ and $\gamma\in \mathbb{E}_T^\sigma\cap\overline{B_M}$,
	\begin{equation*}
	\vertiii{K_T(\gamma)}_{\mathbb{E}_T}\leq \vertiii{K_T(\gamma)-K_T(\mathcal{L}\sigma)}_{\mathbb{E}_T}+\vertiii{K_T(\mathcal{L}\sigma)}_{\mathbb{E}_T}\leq\nicefrac{M}{2}+\nicefrac{M}{2}= M\,.
	\end{equation*}
\end{proof}
\begin{thm}\label{short time existence}
	Let $p\in(3,\infty)$ and $\sigma$ be an admissible initial parametrisation. There exists a positive time $\widetilde{T}\left(\sigma\right)$ depending on $\min_{i\in\{1,2,3\},x\in[0,1]}\vert\sigma^i_x(x)\vert$ and $\left\lVert\sigma\right\rVert_{W_p^{2-\nicefrac{2}{p}}\left((0,1);(\mathbb{R}^n)^3\right)}$ such that for all $\boldsymbol{T}\in (0,\widetilde{T}(\sigma)]$ the system~\eqref{problema} has a solution $\mathcal{E}\sigma$ in 
	\begin{equation*}
	\boldsymbol{E}_{\boldsymbol{T}}=W_p^{1}\left((0,\boldsymbol{T});L_p\left((0,1);(\mathbb{R}^n)^3\right)\right)\cap L_p\left((0,\boldsymbol{T});W_p^2\left((0,1);(\mathbb{R}^n)^3\right)\right)
	\end{equation*}
	which is unique in $\boldsymbol{E}_{\boldsymbol{T}}\cap\overline{B_M}$ with
	\begin{equation*}
	M:=2\max\left\{\sup_{T\in(0,1]}\vertiii{L_T^{-1}}_{\mathcal{L}\left(\mathbb{F}_T,\mathbb{E}_T\right)},1\right\}\max\left\{\vertiii{\mathcal{L}\sigma}_{\boldsymbol{E}_1},\vertiii{\left(N_{1}^1(\mathcal{L}\sigma),\sigma(1),0,N_{1}^2(\mathcal{L}\sigma),\sigma\right)}_{\mathbb{F}_1}\right\}
	\end{equation*}
	where $\mathcal{L}\sigma:=L_1^{-1}\left(0,\sigma(1),0,0,\sigma\right)$ denotes the extension defined in Lemma~\ref{trace} with $T=1$.
\end{thm}
\begin{proof}
	Let $M$ and $\widehat{T}\left(\boldsymbol{c},M\right)$ be as in Proposition~\ref{selfmapping} and let $\boldsymbol{T}\in(0,\widehat{T}\left(\boldsymbol{c},M\right)]$. The fixed points of the mapping $K_{\boldsymbol{T}}$ in $\mathbb{E}_{\boldsymbol{T}}^\sigma\cap\overline{B_M}$
	are precisely the solutions of the system~\eqref{problema} in the space $\boldsymbol{E}_{\boldsymbol{T}}\cap\overline{B_M}$. As $K_{\boldsymbol{T}}$ is a contraction of the complete metric space $\mathbb{E}_{\boldsymbol{T}}^\sigma\cap\overline{B_M}$, existence and uniqueness of a solution follow from the Contraction Mapping Principle.
\end{proof}

\begin{proof}[Proof of Theorem~\ref{existenceanalyticprob}]
	This follows from Theorem~\ref{short time existence} where the appropriate time $\boldsymbol{T}$ and radius $M$ are specified.
\end{proof}

\subsection{Existence and uniqueness of solutions to the motion by curvature}\label{realexistence}

Now that we obtained existence and uniqueness of solutions
to the Special Flow~\eqref{problema} we can come back to our geometric problem.

\begin{thm}[Local existence of the motion by curvature]\label{existencegeopro}
Let $p\in(3,\infty)$ and $\mathbb{T}_0$ be a
geometrically admissible initial Triod. 
Then
there exists $T>0$ such that there exists 
a solution to the motion by curvature in $[0,T]$ with initial datum $\mathbb{T}_0$ as defined in Definition~\ref{geosolution} 
which can be described by one parametrisation
in the whole time interval $[0,T]$.
\end{thm}
\begin{proof}
By Definition~\ref{geomadm} the geometrically admissible initial Triod
$\mathbb{T}_0$ admits a parametrisation $\sigma=(\sigma^1,\sigma^2,\sigma^3)$
that is an admissible initial 
parametrisation for the Special Flow.
Theorem~\ref{existenceanalyticprob} implies that there exists $\boldsymbol{T}>0$ and a solution
$\mathcal{E}\sigma\in \boldsymbol{E}_{\boldsymbol{T}}$ to the Special Flow~\eqref{problema} in $[0,\boldsymbol{T}]$
with $(\mathcal{E}\sigma)^i(0)=\sigma^i$.  
Then by Definition~\ref{geosolution} $\mathbb{T}=\bigcup_{i=1}^3(\mathcal{E}\sigma)^i([0,\boldsymbol{T}]\times[0,1])$
is a solution to the motion by curvature in $[0,\boldsymbol{T}]$ with initial datum $\mathbb{T}_0$.
\end{proof}

\begin{lem}[A composition property]\label{stableunderrepara}
	Let $p\in(3,\infty)$, $T$ be positive and 
	\begin{equation*}
	f,g\in L_p\left((0,T);W_p^2\left((0,1)\right)\right)\cap W_p^1\left((0,T);L_p((0,1))\right)
	\end{equation*}
	be such that for every $t\in[0,T]$ the map $g(t,\cdot):[0,1]\to[0,1]$ is a $C^1$--diffeomorphism. Then the map $h(t,x):=f(t,g(t,x))$ lies in $L_p\left((0,T);W_p^2\left((0,1)\right)\right)\cap W_p^1\left((0,T);L_p((0,1))\right)$ and all derivatives can be calculated by the chain rule.
\end{lem}
\begin{proof}
	This can be shown with similar arguments as in~\cite[Lemma 5.3]{garckemenzelpludawillmore} using the embedding in Theorem~\ref{embeddingBUC}.
\end{proof}

\begin{thm}[Local uniqueness of the motion by curvature]\label{geouniquenesslocal}
Let $p\in(3,\infty)$, $T, \widetilde{T}>0$, $\mathbb{T}_0$ be a geometrically admissible initial Triod
	and $(\mathbb{T}(t))$, $(\widetilde{\mathbb{T}}(t))$ be two solutions
	to the motion by curvature
	with initial datum $\mathbb{T}_0$
	in $[0,T]$ and $[0,\widetilde{T}]$, respectively, as defined in Definition~\ref{geosolution}.
	Then there exists a positive time $\widehat{T}\leq \min\{T,\widetilde{T}\}$
	such that
	$\mathbb{T}(t)=\widetilde{\mathbb{T}}(t)$
	for all $t\in [0,\widehat{T}]$. 
\end{thm}

\begin{proof}
Let $\mathbb{T}_0$ be a geometrically admissible initial Triod with regular parametrisation $\sigma\in W_p^{2-\nicefrac{2}{p}}\left((0,1);(\mathbb{R}^n)^3\right)$. 
Then $\sigma$ is an admissible initial value for the Special Flow~\eqref{problema} and Theorem~\ref{existenceanalyticprob} yields that there exists $\boldsymbol{T}>0$
and a solution 
$\mathcal{E}\sigma=((\mathcal{E}\sigma)^1,(\mathcal{E}\sigma)^2,(\mathcal{E}\sigma)^3)\in\boldsymbol{E}_{\boldsymbol{T}}$ of~\eqref{problema} with initial datum $\sigma$ which is unique in $\boldsymbol{E}_{\boldsymbol{T}}\cap\overline{B_M}$ with $M$ as in Theorem~\ref{short time existence}. In particular, $\mathbb{T}(t):=(\mathcal{E}\sigma)\left(t,[0,1]\right)$ defines a solution to the motion by curvature~\eqref{systemtriod} in $[0,\boldsymbol{T}]$ with initial datum $\mathbb{T}_0$. Suppose that there is another solution $(\widetilde{\mathbb{T}}(t))$ to the motion by curvature in the sense of Definition~\ref{geosolution} with initial datum $\mathbb{T}_0$ in a time interval 
$[0,\widetilde{T}]$ for some positive $\widetilde{T}$. By possibly decreasing the time of existence $\widetilde{T}$ we may assume that there exists one parametrisation $\widetilde{\gamma}\in\boldsymbol{E}_{\widetilde{T}}$ 
for the evolution $(\widetilde{\mathbb{T}}(t))$
in the whole time interval $[0,\widetilde{T}]$.

We show that there exists a family of time dependent diffeomorphisms $\psi^i(t):[0,1]\to[0,1]$ with $t\in[0,\widehat{T}]$ for some $\widehat{T}\leq\min\{\widetilde{T},\boldsymbol{T}\}$ such that the equality
\begin{equation*}
\widetilde{\gamma}^i(t,\psi^i(t,x))=(\mathcal{E}\sigma)^i(t,x)
\end{equation*}
holds in the space $\boldsymbol{E}_{\widehat{T}}$. In order to make use of the uniqueness assertion in Theorem~\ref{existenceanalyticprob} we construct the reparametrisations $\psi=\left(\psi^1,\psi^2,\psi^3\right)$ in such a way that the functions $(t,x)\mapsto \widetilde{\gamma}^i(t,\psi^i(t,x))$ are a solution to the Special Flow in $\boldsymbol{E}_{\widehat{T}}$ with initial datum $\sigma$.  

One easily shows that there exist unique diffeomorphisms $\psi^i_0:[0,1]\to[0,1]$, $i\in\{1,2,3\}$, of regularity $\psi^i_0\in W_p^{2-\nicefrac{2}{p}}\left((0,1);\mathbb{R}\right)$ such that $\psi^i_0(0)=0$, $\psi^i_0(1)=1$ and $\widetilde{\gamma}^i(0,\psi^i_0(x))=\sigma^i(x)$.
Taking into account the special tangential velocity in~\eqref{problema} (formal) differentiation shows that the reparametrisations $\psi^i$ need to satisfy the following boundary value problem:
\begin{equation}\label{systemrepara}
\begin{cases}
\begin{array}{ll}
\psi^i_t(t,x)&=\frac{\psi_{xx}^{i}\left(t,x\right)}{\left|\widetilde{\gamma}_{x}^{i}\left(t,\psi^i(t,x)\right)\right|^{2}\psi^i_x(t,x)^2}-\frac{\left\langle\widetilde{\gamma}_t^i(t,\psi^i(t,x)),\widetilde{\gamma}^i_x(t,\psi^i(t,x))\right\rangle}{\vert\widetilde{\gamma}^i_x(t,\psi^i(t,x))\vert^2}\\
&\,\,\,\,\,+\frac{1}{\left|\widetilde{\gamma}_{x}^{i}\left(t,
	\psi^i(t,x)\right)\right|}
\left\langle\frac{\widetilde{\gamma}_{xx}^{i}\left(t,\psi^i(t,x)\right)}{\left|
	\widetilde{\gamma}_{x}^{i}\left(t,\psi^i(t,x)\right)\right|^{2}}\,,\,
\frac{\widetilde{\gamma}_x^i(t,\psi^i(t,x))}{\left|\widetilde{\gamma}_{x}^{i}\left(t,
	\psi^i(t,x)\right)\right|}\right\rangle\,,
 \\
\psi^i(t,0)&=0\,,
\\
\psi^i(t,1)&=1\,,
\\
\psi^i(0,x)&=\psi_0^i(x)\,.\\
\end{array}
\end{cases}
\end{equation}
Lemma~\ref{existencerepara} yields that there exists a solution
\begin{equation*}
\psi=\left(\psi^1,\psi^2,\psi^3\right)\in W_p^1((0,\widehat{T});L_p((0,1);\mathbb{R}^3))\cap L_p((0,\widehat{T});W_p^2((0,1);\mathbb{R}^3))
\end{equation*}
to system~\eqref{systemrepara} for some $\widehat{T}\leq\min\{\widetilde{T},\boldsymbol{T}\}$ such that $\psi^i(t):[0,1]\to[0,1]$ is a diffeomorphism for every $t\in[0,\widehat{T}]$. Then Lemma~\ref{stableunderrepara} implies that the composition $(t,x)\mapsto\widetilde{\gamma}^i(t,\psi^i(t,x))$ lies in $\boldsymbol{E}_{\widehat{T}}$ and by construction, it is a solution to the Special Flow. We may now argue as in the 
proof of~\cite[Theorem 5.4]{garckemenzelpludawillmore} to obtain that $(t,x)\mapsto(\mathcal{E}\sigma)^i(t,x)$ and $(t,x)\mapsto\widetilde{\gamma}^i(t,\psi^i(t,x))$ coincide in $\boldsymbol{E}_{\widehat{T}}$. In particular, the networks $\mathbb{T}(t)$ and $\widetilde{\mathbb{T}}(t)$ coincide for all $t\in[0,\widehat{T}]$.
\end{proof}

\begin{lem}\label{existencerepara}
	Let $p\in(3,\infty)$, $\psi_0=(\psi_0^1,\psi_0^2,\psi_0^3)\in W_p^{2-\nicefrac{2}{p}}\left((0,1);\mathbb{R}^3\right)$ with $\psi_0^i:[0,1]\to[0,1]$ a diffeomorphism with $\psi_0^i(0)=0$, $\psi_0^i(1)=1$, $\widetilde{T}>0$ and  $\widetilde{\gamma}\in\mathbb{E}_{\widetilde{T}}$ be such that $\widetilde{\gamma}^i_x(x)\neq 0$ for all $x\in[0,1]$. Then there exists a time $\widehat{T}\in (0,\widetilde{T}]$ and a solution 
	\begin{equation*}
	\psi=\left(\psi^1,\psi^2,\psi^3\right)\in W_p^1((0,\widehat{T});L_p((0,1);\mathbb{R}^3))\cap L_p((0,\widehat{T});W_p^2((0,1);\mathbb{R}^3))
	\end{equation*}
	to system~\eqref{systemrepara} such that $\psi^i(t):[0,1]\to[0,1]$ is a diffeomorphism for every $t\in[0,\widehat{T}]$.
\end{lem}
\begin{proof}
	We observe that the right hand side of the motion equation in system~\eqref{systemrepara} contains terms of the form $f^i(t,\psi^i(t,x))$ with $f^i\in L_p\left((0,T);L_p((0,1))\right)$. To remove this dependence it is convenient to consider the associated problem for the inverse diffeomorphisms $\xi=(\xi^1,\xi^2,\xi^3)$ given by $\xi^i(t):=\psi^i(t)^{-1}$. Indeed suppose that $\psi\in W^{1,2}_p((0, \widetilde{T})\times (0,1);\mathbb{R}^3)$
	is a solution to~\eqref{systemrepara} with $\psi^i(t):[0,1]\to[0,1]$ 
	a $C^1$--diffeomorphism. Similar arguments as in~\cite[Lemma 5.3]{garckemenzelpludawillmore} show that also $\xi$ is of class $W^{1,2}_p((0, \widetilde{T})\times (0,1);\mathbb{R}^3)$.
	Moreover, the differentiation rules 
	\begin{align*}
	\xi^i_y(t,y)&=\psi^i_x(t,\xi^i(t,y))^{-1}\,,\\
	\xi^i_{yy}(t,y)&=-\xi^i_y(t,y)^3\psi^i_{xx}(t,\xi^i(t,y))
	\end{align*}
	yield the evolution equation
	\begin{align*}
	\xi^i_t(t,y)
	=&-\psi^i_t(t,\xi^i(t,y))\xi_y^i(t,y)\\
	=&-\frac{\psi_{xx}^{i}\left(t,\xi^i(t,y)\right)}{\left|\widetilde{\gamma}_{x}^{i}\left(t,y\right)\right|^{2}}
	\xi^i_y(t,y)^3
	+\frac{\left\langle\widetilde{\gamma}_t^i(t,y),\widetilde{\gamma}^i_x(t,y)\right\rangle}{\vert\widetilde{\gamma}^i_x(t,y)\vert^2}\xi^i_y(t,y)\\
	&-\frac{\xi^i_y(t,y)}{\left|\widetilde{\gamma}_{x}^{i}\left(t,y\right)\right|}
	\left\langle\frac{\widetilde{\gamma}_{xx}^{i}\left(t,y\right)}{\left|
		\widetilde{\gamma}_{x}^{i}\left(t,y\right)\right|^{2}}\,,\,
	\frac{\widetilde{\gamma}_x^i(t,y)}{\left|\widetilde{\gamma}_{x}^{i}\left(t,y\right)\right|}\right\rangle\,,
	\end{align*}
	and in conclusion the following system for $\xi$:	
	\begin{equation}\label{systemreparainverse}
	\begin{cases}
	\begin{array}{ll}
	\xi^i_t(t,y)&
	=\frac{\xi^i_{yy}(t,y)}{\left|\widetilde{\gamma}_{x}^{i}\left(t,y\right)\right|^{2}}
	+\frac{\left\langle\widetilde{\gamma}_t^i(t,y),\widetilde{\gamma}^i_x(t,y)\right\rangle}{\vert\widetilde{\gamma}^i_x(t,y)\vert^2}\xi^i_y(t,y)-\frac{\xi^i_y(t,y)}{\left|\widetilde{\gamma}_{x}^{i}\left(t,y\right)\right|}
	\left\langle\frac{\widetilde{\gamma}_{xx}^{i}\left(t,y\right)}{\left|
		\widetilde{\gamma}_{x}^{i}\left(t,y\right)\right|^{2}}\,,\,
	\frac{\widetilde{\gamma}_x^i(t,y)}{\left|\widetilde{\gamma}_{x}^{i}\left(t,y\right)\right|}\right\rangle\,,
	\\
	\xi^i(t,0)&=0\,,
	\\
	\xi^i(t,1)&=1\,,
	\\
	\xi^i(0,y)&=(\psi_0^i)^{-1}(y)\\
	\end{array}
	\end{cases}
	\end{equation}
	for all $t\in [0,\widetilde{T}]$, $y\in[0,1]$.
	We observe that the boundary value problem~\eqref{systemreparainverse} has a very similar structure as the Special Flow. Analogous arguments as in the proof of Theorem~\ref{existenceanalyticprob} allow us to conclude that there exists a solution $\xi\in W^{1,2}_p((0,\widehat{T})\times (0,1);(\mathbb{R}^2)^3)$
	to~\eqref{systemreparainverse} with $\widehat{T}\in(0,\widetilde{T}]$ such that for $t\in[0,\widehat{T}]$ the map
	$\xi^i(t):[0,1]\to[0,1]$ is a $C^1$--diffeomorphism. Reversing the above argumentation yields that the inverse functions $\psi^i(t):=\xi^i(t)^{-1}$ solve~\eqref{systemrepara}
	and possess the desired properties.  
\end{proof}

\begin{thm}[Geometric uniqueness of the motion by curvature]\label{geouniqueness}
	Let $p\in(3,\infty)$, $\mathbb{T}_0$ be a geometrically admissible initial Triod and $T$ be positive. 
	Solutions to the motion by curvature in $[0,T]$ with initial datum $\mathbb{T}_0$ are geometrically unique in the sense that given any two solutions $(\mathbb{T}(t))$ and $(\widetilde{\mathbb{T}}(t))$ to the motion by curvature in the time interval $[0,T]$ with initial datum $\mathbb{T}_0$ the networks $\mathbb{T}(t)$ and $\widetilde{\mathbb{T}}(t)$ coincide for all $t\in[0,T]$.
\end{thm}
\begin{proof}
	Let $(\mathbb{T}(t))$ and $(\widetilde{\mathbb{T}}(t))$ be two solutions to the motion by curvature in $[0,T]$ with initial datum $\mathbb{T}_0$. Suppose by contradiction that the set
	\begin{equation*}
	\mathcal{S}:=\left\{t\in [0,T]:\mathbb{T}(t)\neq\widetilde{\mathbb{T}}(t)\right\}
	\end{equation*}
	is non-empty and let $t^*:=\inf\mathcal{S}$. As $\mathcal{S}$ is an open subset of $[0,T]$, we have $t^*\in [0,T)$ and $\mathbb{T}(t^*)=\widetilde{\mathbb{T}}(t^*)$. The Triod $\mathbb{T}(t^*)$ is geometrically admissible and both $t\mapsto\mathbb{T}(t^*+t)$ and $t\mapsto\widetilde{\mathbb{T}}(t^*+t)$ are solutions to the motion by curvature in the time interval $[0,T-t^*]$ with initial datum $\mathbb{T}(t^*)$. Theorem~\ref{geouniquenesslocal} yields that there exists a time $\widehat{T}\in(0,T-t^*]$ such that for all $t\in[0,\widehat{T}]$, $\mathbb{T}(t^*+t)=\widetilde{\mathbb{T}}(t^*+t)$ which contradicts the definition of $t^*$.
\end{proof}

\begin{defn}[Maximal solutions to the motion by curvature]\label{maximalsolution}
Let $p\in(3,\infty)$ and $\mathbb{T}_0$ be a geometrically admissible initial network.
A time--dependent family of Triods $\left(\mathbb{T}(t)\right)_{t\in[0,T)}$ with $T\in(0,\infty)\cup\{\infty\}$
is a \textit{maximal solution to
the motion by curvature} in $[0,T)$
with initial datum $\mathbb{T}_0$ 
if it is a solution 
(in the sense of Definition~\ref{geosolution}) 
in $[0,\hat{T}]$ for all $\hat{T}<T$ 
and if there does not exist a solution 
$(\widetilde{\mathbb{T}}(\tau))$ 
to the motion by curvature in the sense of Definition~\ref{geosolution}
in $[0,\widetilde{T}]$ with $\widetilde{T}\geq T$ and such that $\mathbb{T}=\widetilde{\mathbb{T}}$ in $[0,T)$.
In this case the time $T$ is called \textit{maximal time of existence} and is denoted by $T_{max}$.
\end{defn}

\begin{prop}[Existence and uniqueness of maximal solutions]\label{exmax}
Let $p\in(3,\infty)$ and $\mathbb{T}_0$ be a geometrically admissible initial network.
There exists a maximal solution to the motion by curvature with initial datum $\mathbb{T}_0$ which is geometrically unique.
\end{prop}
\begin{proof}
Given an admissible Triod $\mathbb{T}_0$ we let
\begin{align*}
T_{max}:=&\sup\left\{T>0: \text{ there exists a solution }(\mathbb{T}^T(t)) \text{ to the motion by curvature in } [0,T]\right.\\
 &\qquad \quad \,\,\,\, \qquad \left.\text{ with initial datum }\mathbb{T}_0\right\}\,.
\end{align*}
Theorem~\ref{existencegeopro} yields $T_{max}\in(0,\infty)\cup\{\infty\}$. Given any $t\in[0,T_{max})$ we may consider a solution $\mathbb{T}^T$ with $T\in(t,T_{max})$ to the motion by curvature in $[0,T]$ with initial datum $\mathbb{T}_0$ and set
\begin{equation*}
\mathbb{T}(t):=\mathbb{T}^T(t)\,.
\end{equation*} 
We note that $\mathbb{T}$ is well-defined on $[0,T_{max})$ as any two solutions $\mathbb{T}^{T_1}$ and $\mathbb{T}^{T_2}$ with $T_1$, $T_2\in[0,T_{max})$ to the motion by curvature with initial datum $\mathbb{T}_0$ coincide on their common interval of existence by Theorem~\ref{geouniqueness}. One easily verifies that $(\mathbb{T}(t))_{t\in[0,T_{max})}$ satisfies the properties of a maximal solution stated in Definition~\ref{maximalsolution}. Indeed, if there existed a solution $\widetilde{\mathbb{T}}(\tau)$ to the motion by curvature in $[0,\widetilde{T}]$ for $\widetilde{T}\geq T_{max}$, Theorem~\ref{existencegeopro} would imply the existence of a solution with initial datum $\widetilde{\mathbb{T}}(\widetilde{T})$ in a time interval $[0,\delta]$, $\delta >0$. This would yield the existence of a solution in the time interval $[0,\widetilde{T}+\delta]$ with initial datum $\mathbb{T}_0$ contradicting the definition of $T_{max}$. 
The uniqueness assertion follows from Theorem~\ref{geouniqueness}.
\end{proof}

\section{Smoothness of the Special Flow}\label{parabolicsmoothingresult}

This section is devoted to prove that solutions to the Special Flow are smooth for positive times. Heuristically, this regularisation effect is due to the parabolic nature of the problem. The basic idea of the proof is based on the so called parameter trick which is due to Angenent~\cite{angen3} and has been generalized to 
several situations~\cite{lunardi1,lunardi2,Prusssimonett}. 
However, due to the fully non-linear boundary condition
\begin{equation*}
\sum_{i=1}^{3} \frac{\gamma^i_x(t,0)}{\vert\gamma^i_x(t,0)\vert}=0
\end{equation*}
the Special Flow is not treated in the above mentioned results.
An adaptation of the parameter trick that allows to treat fully non-linear boundary terms 
is presented in~\cite{michi}.
We follow~\cite[Section 4]{michi}
modifying the arguments for the application in our Sobolev setting.

In the following we let $\mathcal{E}\sigma\in\boldsymbol{E}_T$ be a solution to the Special Flow on $[0,T]$, $T>0$, with initial datum $\sigma \in W_p^{2-\nicefrac{2}{p}}\left((0,1);(\mathbb{R}^n)^3\right)$.

The key idea to apply Angenent's parameter trick lies in an implicit function type argument which itself relies on the invertibility of the linearisation of the Special Flow in the solution $\mathcal{E}\sigma$. Thus, the linear analysis from Subsection~\ref{linearcompcond} will not be enough to apply this method. So before we can actually start we have to generalise Theorem~\ref{exlin}.

\begin{defn}\label{DefinitionLinearisationinEsigma}
	We consider the full linearisation of system~\eqref{problema} around $\mathcal{E}\sigma$ which gives
	\begin{align}\label{EquationLinearisationinEsigma}
	\begin{cases}
	\begin{array}{rl}
	\gamma^i_{t}(t,x)-\frac{1}{\left|(\mathcal{E}\sigma)^i_{x}(t,x)\right|^{2}}\,\gamma^i_{xx}(t,x)-2\frac{(\mathcal{E}\sigma)^i_{xx}(t,x)\left\langle\gamma^i_x(t,x),(\mathcal{E}\sigma)^i_x(t,x)\right\rangle}{\vert(\mathcal{E}\sigma)^i_x(t,x)\vert^4}&=f^i(t,x)\,,\\
	\gamma(t,1)&=\eta(t)\,,
	\\
	\gamma^1\left(t,0\right)-\gamma^{2}\left(t,0\right)&=0\,,\\
	\gamma^2(t,0)-\gamma^{3}\left(t,0\right)&=0\,,  \\
	-\sum_{i=1}^3 \left(\frac{\gamma^i_x(t,0)}{\vert (\mathcal{E}\sigma)^i_x(t,0)\vert}
	-\frac{(\mathcal{E}\sigma)^i_x(t,0)\left\langle
		\gamma^i_x(t,0),(\mathcal{E}\sigma)^i_x(t,0)\right\rangle}{\vert (\mathcal{E}\sigma)^i_x(t,0)\vert^3}\right)&=b(t)\,,\\
	\gamma\left(0,x\right)&=\psi\left(x\right)  \,.\\
	\end{array}
	\end{cases}
	\end{align}
	
	%
	%
	Here $\psi$ is an admissible initial value with respect to the given right hand side $\eta$ and $b$. For $\gamma\in \boldsymbol{E}_T$ we define $\mathcal{A}_{T,\mathcal{E}}(\gamma)\in L_p\left((0,T);L_p((0,1);(\mathbb{R}^n)^3)\right)$ by
	\begin{equation*}
	\left(\mathcal{A}_{T,\mathcal{E}}(\gamma)\right)^i:=\frac{1}{\left|(\mathcal{E}\sigma)^i_{x}(t,x)\right|^{2}}\,\gamma^i_{xx}(t,x)+2\frac{(\mathcal{E}\sigma)^i_{xx}(t,x)\left\langle\gamma^i_x(t,x),(\mathcal{E}\sigma)^i_x(t,x)\right\rangle}{\vert(\mathcal{E}\sigma)^i_x(t,x)\vert^4}\,.
	\end{equation*}
\end{defn}
\begin{defn}[The linearised boundary operator]
	Let $T>0$ and 
	\begin{equation*}
	\mathcal{B}_{T,\mathcal{E}}:\boldsymbol{E}_T=W_p^{1,2}\left((0,T)\times(0,1);(\mathbb{R}^n)^3\right)\to W_p^{1-\nicefrac{1}{2p}}\left((0,T);(\mathbb{R}^n)^5\right)\times W_p^{\nicefrac{1}{2}-\nicefrac{1}{2p}}\left((0,T);\mathbb{R}^n\right)
	\end{equation*}
	be the linearised boundary operator induced by the linearisation in $\mathcal{E}\sigma$, i.e.,
	\begin{equation*}
	\mathcal{B}_{T,\mathcal{E}}(\gamma)=
	\begin{pmatrix}
	\gamma(\cdot,1) \\
	\gamma^1(\cdot,0)-\gamma^2(\cdot,0)\\
	\gamma^2(\cdot,0)-\gamma^3(\cdot,0)\\
	-\sum_{i=1}^3 \frac{\gamma^i_x(\cdot,0)}{\vert (\mathcal{E}\sigma)^i_x(\cdot,0)\vert}
	-\frac{(\mathcal{E}\sigma)^i_x(\cdot,0)\left\langle
		\gamma^i_x(\cdot,0),(\mathcal{E}\sigma)^i_x(\cdot,0)\right\rangle}{\vert (\mathcal{E}\sigma)^i_x(\cdot,0)\vert^3}
	\end{pmatrix}.
	\end{equation*}
	Moreover we let 
	\begin{align*}
	\boldsymbol{X}_T   & :=\ker(\mathcal{B}_{T,\mathcal{E}})\,.
	\end{align*}
\end{defn}
As $\mathcal{B}_{T,\mathcal{E}}$ is continuous, the space $\boldsymbol{X}_T$ is a closed subspace of $\boldsymbol{E}_T$ and thus a Banach space.
\begin{rem}[Existence analysis for~\eqref{EquationLinearisationinEsigma}]\label{RemarkExistenceAnalysisLinearisationinEsigma}
	Note that the compatibility conditions in Definition~\ref{linearcompcond} for system~\eqref{linsys} are precisely the same as the ones for \eqref{EquationLinearisationinEsigma} due to the fact that $\mathcal{B}_{T,\mathcal{E}}\big|_{t=0}$ equals the original linearisation. Also, with the same arguments as in the proof of Lemma \ref{LemmaLopatinskisShapiroconditions} we can derive the Lopatinskii-Shapiro conditions for $\mathcal{B}_{T,\mathcal{E}}$. Therefore, the result from Theorem \ref{exlin} holds also for problem \eqref{EquationLinearisationinEsigma}. For $\gamma\in\boldsymbol{E}_T$ we write
	\begin{equation*}
	L_{T,\mathcal{E}}(\gamma):=\begin{pmatrix}
	\gamma_t-A_{T,\mathcal{E}}(\gamma)\\
	\mathcal{B}_{T,\mathcal{E}}(\gamma)\\
	\gamma_{|t=0}
	\end{pmatrix}\,.
	\end{equation*}
\end{rem}
With the previous considerations we have the basics to start the work on the parameter trick. As a first step we have to construct a parametrisation of the non-linear boundary conditions over the linear boundary conditions. We need to do this as we cannot have the non-linear boundary operator to be part of the operator used in the parameter trick due to technical reasons with the compatibility conditions.  

In the following lemma we construct a partition of the solution space $\boldsymbol{E}_T=\boldsymbol{X}_T\oplus \boldsymbol{Z}_T$.
\begin{lem}
	Let $T>0$. There exists a closed subspace $\boldsymbol{Z}_T$ of $\boldsymbol{E}_T$ such that $\boldsymbol{E}_T=\boldsymbol{X}_T\oplus\boldsymbol{Z}_T$.
\end{lem}
\begin{proof}
	Firstly, we consider the space 
	\begin{equation*}
	\overline{Z}^{1}_T:=\left\{\mathfrak{b}\in W_p^{1-\nicefrac{1}{2p}}\left((0,T);(\mathbb{R}^n)^5\right)\times W_p^{\nicefrac{1}{2}-\nicefrac{1}{2p}}\left((0,T);\mathbb{R}^n\right): \mathfrak{b}_{|t=0}=0\right\}.
	\end{equation*}
	We notice that $f=0$, $\mathfrak{b}\in\overline{Z}^1_T$, $\psi=0$ is a suitable right hand side for problem~\eqref{EquationLinearisationinEsigma}. Hence for every $\mathfrak{b}\in \overline{Z}_T^1$ there exists a unique solution $L_{T,\mathcal{E}}^{-1}\left(0,\mathfrak{b},0\right)\in \boldsymbol{E}_T$ to~\eqref{EquationLinearisationinEsigma} and the space $Z_T^1:=L_{T,\mathcal{E}}^{-1}\left((0,\overline{Z}_T^1,0)\right)$ is a closed subspace of $\boldsymbol{E}_T$.
	
	Next we define the space 
	\begin{equation*}
	\overline{Z}^2:=(\mathbb{R}^n)^5\times\mathbb{R}^n\,.
	\end{equation*}
	Given $\tilde{b}\in\overline{Z}^2$ the elliptic system $\tilde{L}\eta=(0,\tilde{b})$ defined by
	\begin{equation}\label{ellipticsystem}
	\begin{cases}
	\begin{array}{rl}
	-\frac{1}{\left|\sigma^i_{x}(x)\right|^{2}}\,\eta^i_{xx}(x)&=0\,,\qquad \quad x\in(0,1)\,,i\in\{1,2,3\}\,,\\
	\eta^1(1)     &=\tilde{b}^1\,,\\
	\eta^2(1)     &=\tilde{b}^2\,,\\
	\eta^3(1)     &=\tilde{b}^3\,,\\
	\eta^1(0)-\eta^{2}(0)&=\tilde{b}^4\,,\\
	\eta^2(0)-\eta^{3}(0)&=\tilde{b}^5\,,  \\
	-\sum_{i=1}^3 \left(\frac{\eta^i_x(0)}{\vert \sigma^i_x(0)\vert}
	-\frac{\sigma^i_x(0)\left\langle
		\eta^i_x(0),\sigma^i_x(0)\right\rangle}{\vert \sigma^i_x(0)\vert^3}\right)&=\tilde{b}^6\,,
	\end{array}
	\end{cases}
	\end{equation}
	has a unique solution $\eta\in W_p^{2}\left((0,1);(\mathbb{R}^n)^3\right)$ which we denote by $\tilde{L}^{-1}(0,\tilde{b})$.
	This is guaranteed due to the results in~\cite{agmondouglisNi1959estimates} and the fact that the boundary operator fulfils the Lopatinskii-Shapiro conditions according to Lemma~\ref{LemmaLopatinskisShapiroconditions}. The space $\tilde{L}^{-1}(0,\overline{Z}^2)$ is a closed subspace of $W_p^{2}\left((0,1);(\mathbb{R}^n)^3\right)$ due to continuity of the solution operator which is guaranteed by the energy estimates 
	in~\cite{agmondouglisNi1959estimates}. Extending every function in $\tilde{L}^{-1}(0,\overline{Z}^2)$ constantly in time we can view $\tilde{L}^{-1}(0,\overline{Z}^2)$ as a closed subspace of $\boldsymbol{E}_T$. This space will be denoted by $Z^2_T$.
	It is straightforward to check that $Z_T^1\cap Z_T^2=\{0\}$ which allows us to define $\boldsymbol{Z}_T$ as the  subspace of $\boldsymbol{E}_T$ given by
	\begin{equation*}
	\boldsymbol{Z}_T:=Z_T^1\oplus Z_T^2\,.
	\end{equation*}
	Note that $\boldsymbol{Z}_T$ is indeed a closed subspace which one sees as follows. Suppose that $$(z_n)_{n\in\N}=(z_n^1+z_n^2)_{n\in\N}\subset \boldsymbol{Z}_T$$ is a convergent sequence in $\boldsymbol{E}_T$. 
	
	Due to $\boldsymbol{E}_T\hookrightarrow C([0,T]; C^{1+\alpha}([0,1];(\R^n)^3))$ for $\alpha\in\left(0,1-\nicefrac{3}{p}\right]$ according to Theorem~\ref{embeddingBUC} we may conclude that the sequence $(z_n\big|_{t=0})_{n\in\N}=(z_n^2\big|_{t=0})_{n\in\N}$ converges in $C^{1+\alpha}([0,1];(\R^n)^3)$. In particular, this yields the convergence of the boundary data needed for the elliptic system we used to construct $z^2_n$. Continuity of the elliptic solution operator then implies that $(z_n^2\big|_{t=0})_{n\in\N}$ converges in $W^2_p((0,1);(\R^n)^3)$. Due to its constant extension in time we see that $(z_n^2)_{n\in\N}$ converges in $\boldsymbol{E}_T$ to a limit $z^2$ which is also in $Z^2_T$ being a closed subspace of $\boldsymbol{E}_T$. Then $(z_n^1)_{n\in\N}=(z_n)_{n\in\N}-(z_n^2)_{n\in\N}$ converges in $\boldsymbol{E}_T$ as sum of two convergent sequences to an element $z^1$ of the closed space $Z_T^1$. We conclude that $(z_n)_{n\in\N}$ converges to $z^1+z^2\in \boldsymbol{Z}_T$ which shows that $\boldsymbol{Z}_T$ is closed.\\	
	It remains to prove that $\boldsymbol{X}_T\cap \boldsymbol{Z}_T=\{0\}$ and $\boldsymbol{E}_T= \boldsymbol{X}_T+\boldsymbol{Z}_T$. To this end let $\gamma\in \boldsymbol{X}_T\cap \boldsymbol{Z}_T$. By definition of $\boldsymbol{X}_T$ we have $\mathcal{B}_{T,\mathcal{E}}(\gamma)=0$ which implies in particular $\mathcal{B}_{T,\mathcal{E}}(\gamma)_{|t=0}=0$. As $\gamma$ lies in $\boldsymbol{Z}_T$, there exist $z_1\in Z_T^1$, $z_2\in Z_T^2$ with $\gamma=z_1+z_2$. Using that $\mathcal{B}_{T,\mathcal{E}}(z_1)$ lies in $\overline{Z}_T^1$, we observe  
	\begin{equation*}
	0=\mathcal{B}_{T,\mathcal{E}}(z_1+z_2)_{|t=0}=\mathcal{B}_{T,\mathcal{E}}(z_1)_{|t=0}+\mathcal{B}_{T,\mathcal{E}}(z_2)_{|t=0}=\mathcal{B}_{T,\mathcal{E}}(z_2)_{|t=0}\,.
	\end{equation*}
	Due to the uniqueness of the elliptic system \eqref{ellipticsystem} this shows $(z_2)_{|t=0}=0$. 
	By definition of $Z_T^2$ we obtain $z_2=0$. This implies $0=\mathcal{B}_{T,\mathcal{E}}(\gamma)=\mathcal{B}_{T,\mathcal{E}}(z_1)$ which gives $z_1=L_{T,\mathcal{E}}^{-1}(0,0,0)=0$.

	To prove that $\boldsymbol{E}_T=\boldsymbol{X}_T+\boldsymbol{Z}_T$ we let $\gamma\in\boldsymbol{E}_T$.
	We define
		\begin{equation*}
		z_2:=\tilde{L}^{-1}(0,\mathcal{B}_{T,\mathcal{E}}(\gamma)_{|t=0})\in Z_T^2
		\end{equation*}
	viewing $z_2$ as an element of $\boldsymbol{E}_T$ by extending it constantly in time. By definition of the boundary operator in the elliptic system~\eqref{ellipticsystem} and due to $(\mathcal{E}\sigma)_{|t=0}=\sigma$ we have
	\begin{equation*}
	\mathcal{B}_{T,\mathcal{E}}(z_2)_{|t=0}=\mathcal{B}_{T,\mathcal{E}}(\gamma)_{|t=0}\,.
	\end{equation*}
	In particular, $\mathcal{B}_{T,\mathcal{E}}(\gamma)-\mathcal{B}_{T,\mathcal{E}}(z_2)$ lies in $\overline{Z}^1_T$ and we may define
	\begin{equation*}
	z_1:=L_{T,\mathcal{E}}^{-1}\left(0,\mathcal{B}_{T,\mathcal{E}}(\gamma)-\mathcal{B}_{T,\mathcal{E}}(z_2),0\right)\in Z_T^1\,.
	\end{equation*}
	Now it remains to show that $\gamma-z_1-z_2$ lies in $\boldsymbol{X}_T$ which is equivalent to $\mathcal{B}_{T, \mathcal{E}}(\gamma-z_1-z_2)=0$ which follows by construction.
\end{proof}

\begin{lem}[Parametrisation of the nonlinear boundary conditions]\label{LemmaParametrizationofNonLinearBoundaryConditions}
	Let $T>0$. There exists a neighbourhood $U$ of $0$ in $\boldsymbol{X}_T$, a function $\varrho:U\to \boldsymbol{Z}_T$ and a neighbourhood $V$ of $\mathcal{E}\sigma$ in $\boldsymbol{E}_T$ such that
	\begin{equation*}
	\left\{\mathcal{E}\sigma+\boldsymbol{u}+\varrho(\boldsymbol{u})\,:\,\boldsymbol{u}\in U\right\}=\left\{\gamma\in V:\mathcal{G}(\gamma)=0\right\}
	\end{equation*}
	where $\mathcal{G}$ denotes the operator 
	\begin{equation*}
	\gamma\mapsto \mathcal{G}(\gamma):=
	\begin{pmatrix}
	\gamma^1(\cdot,1)-\sigma^1(1)\,\\
	\gamma^2(\cdot,1)-\sigma^2(1)\,\\
	\gamma^3(\cdot,1)-\sigma^3(1)\,\\
	\gamma^1(\cdot,0)-\gamma^2(\cdot,0)\,\\ 
	\gamma^2(\cdot,0)-\gamma^3(\cdot,0)\,\\
	\sum_{i=1}^3\frac{\gamma^i_x(\cdot,0)}{\vert \gamma^i_x(\cdot,0)\vert}\,
	\end{pmatrix}\,.
	\end{equation*}
	Furthermore, it holds that $(D\varrho)_{|0}\equiv 0$. 
\end{lem}
\begin{proof}
	We let 
	\begin{equation*}
	\boldsymbol{Y}_T:=W_p^{1-\nicefrac{1}{2p}}\left((0,T);(\mathbb{R}^n)^5\right)\times W_p^{\nicefrac{1}{2}-\nicefrac{1}{2p}}\left((0,T);\mathbb{R}^n\right)
	\end{equation*}
	and consider the operator
	\begin{align*}
	F:\boldsymbol{X}_T\oplus \boldsymbol{Z}_T&\to \boldsymbol{Y}_T\,,\\
	(\boldsymbol{x},\boldsymbol{z}) &\mapsto \mathcal{G}\left(\mathcal{E}\sigma+\boldsymbol{x}+\boldsymbol{z}\right)\,.
	\end{align*}
	By definition of $\mathcal{E}\sigma$ we have that $F(0,0)=0$. We observe that $DF_{|(0,0)}(0,\boldsymbol{z})=\mathcal{B}_{T,\mathcal{E}}(\boldsymbol{z})$. To apply the implicit function theorem we have to show that 
	\begin{equation*}
	\mathcal{B}_{T,\mathcal{E}}: \boldsymbol{Z}_T\to \boldsymbol{Y}_T
	\end{equation*}
	is an isomorphism. The map is injective as 
	$\ker{\mathcal{B}_{T,\mathcal{E}}}\cap \boldsymbol{Z}_T=\boldsymbol{X}_T\cap \boldsymbol{Z}_T=\{0\}$.
	Given $\boldsymbol{b}\in \boldsymbol{Y}_T$ we let $z_2:=\tilde{L}^{-1}(0,\boldsymbol{b}_{|t=0})\in Z_T^2$
	and
	 $z_1:=L_{T,\mathcal{E}}^{-1}(0,\boldsymbol{b}-\mathcal{B}_{T,\mathcal{E}}(z_2))\in Z^1_T$ and observe that $z_1+z_2\in\boldsymbol{Z}_T$ satisfies
	\begin{equation*}
	\mathcal{B}_{T,\mathcal{E}}(z_1+z_2)=\mathcal{B}_{T,\mathcal{E}}(z_1)+\mathcal{B}_{T,\mathcal{E}}(z_2)=\boldsymbol{b}-\mathcal{B}_{T,\mathcal{E}}(z_2)+\mathcal{B}_{T,\mathcal{E}}(z_2)=\boldsymbol{b}\,.
	\end{equation*}
	The implicit function theorem implies that there exist neighbourhoods $U$ and $W$ of $0$ in $\boldsymbol{X}_T$ and $\boldsymbol{Z}_T$, respectively, and a function $\varrho:U\to W$ with $\varrho(0)=0$ such that for a neighbourhood $\tilde{V}$ of $0$ in $\boldsymbol{E}_T$, it holds
	\begin{equation*}
	\{\boldsymbol{u}+\varrho(\boldsymbol{u}):\boldsymbol{u}\in U\}=\{\boldsymbol{x}+\boldsymbol{z}\in\boldsymbol{E}_T: F(\boldsymbol{x},\boldsymbol{z})=0\}\cap \tilde{V}\,.
	\end{equation*}
	To show that $(D\varrho)_{|0}=0$ we let $\boldsymbol{u}\in \boldsymbol{X}_T$ be arbitrary. Due to $(D\varrho)_{|0}:\boldsymbol{X}_T\to \boldsymbol{Z}_T$ we obtain $(D\varrho)_{|0}\boldsymbol{u}\in \boldsymbol{Z}_T$. Hence it is enough to show that $(D\varrho)_{|0}\boldsymbol{u}$ lies also in $\boldsymbol{X}_T$. To this end we differentiate the identity
	\begin{equation*}
	0=F(\delta \boldsymbol{u},\varrho(\delta \boldsymbol{u}))=\mathcal{G}\left(\mathcal{E}\sigma+\delta \boldsymbol{u}+\varrho(\delta \boldsymbol{u})\right)
	\end{equation*}
	with respect to $\delta$ and obtain
	\begin{equation*}
	0=\frac{\mathrm{d}}{\mathrm{d}\delta}\mathcal{G}\left(\mathcal{E}\sigma+\delta \boldsymbol{u}+\varrho(\delta \boldsymbol{u})\right)_{|\delta=0}=\left(D\mathcal{G}\right)(\mathcal{E}\sigma)(\boldsymbol{u}+(D\varrho)_{|0}\boldsymbol{u})=\mathcal{B}_{T,\mathcal{E}}(\boldsymbol{u}+(D\varrho)_{|0}\boldsymbol{u})\,.
	\end{equation*}
	This implies $\boldsymbol{u}+(D\varrho)_{|0}\boldsymbol{u}\in\ker\mathcal{B}_{T,\mathcal{E}}=\boldsymbol{X}_T$ and thus $(D\varrho)_{|0}\boldsymbol{u}\in \boldsymbol{X}_T$. 
\end{proof}
With this result we can finally start the proof of the parabolic smoothing. We will first derive higher time regularity of the solution (this is actually the classical parameter trick argument by Angenent), and will then get from this higher regularity in space using the parabolic problem and finally start a bootstrap procedure.
\begin{prop}[Higher time regularity of solutions to the Special Flow]\label{propositionhighertimeregularityspecialflow}\ \\
	Let $\mathcal{E}\sigma\in\boldsymbol{E}_T$ be a solution to the Special Flow in $[0,T]$ with $T>0$ and initial value $\sigma\in W_p^{2-\nicefrac{2}{p}}\left((0,1);(\mathbb{R}^n)^3\right)$. 
	Then we have for all $\tilde{t}\in(0,T]$ the increased time regularity
		\begin{equation}\label{EquationHigherTimeRegularity}
		\partial_t(\mathcal{E}{\sigma})\in \boldsymbol{E}_{T}\big|_{[\tilde{t},T]}\,.
		\end{equation}
\end{prop}
\begin{proof}
	We consider the space
	\begin{align*}
	\boldsymbol{I}:=&\left\{\psi\in W_p^{2-\nicefrac{2}{p}}\left((0,1);(\mathbb{R}^n)^3\right): \psi(1)=0\,, \psi^1(0)=\psi^2(0)=\psi^3(0)\,,\right.\\
	&\left. \quad \sum_{i=1}^3\frac{\psi^i_x(0)}{\vert\sigma^i_x(0)\vert}-\frac{\sigma^i_x(0)\left\langle\psi^i_x(0),\sigma^i_x(0)\right\rangle}{\vert\sigma^i_x(0)\vert ^3}=0\right\}\,.
	\end{align*}
	We let $U$, $V$ and $\varrho$ be as in the previous Lemma and define $\overline{\varrho}(\boldsymbol{u}):=\mathcal{E}\sigma+\boldsymbol{u}+\varrho(\boldsymbol{u})$.
	For some small $\varepsilon\in(0,1)$ we consider the map 
	\begin{align*}
	G:(1-\varepsilon,1+\varepsilon)\times \boldsymbol{I}\times \boldsymbol{X}_T&\to \boldsymbol{I}\times L_p\left((0,T)\times(0,1);(\mathbb{R}^n)^3\right)\,,\\
	\left(\lambda,\psi,\boldsymbol{u}\right)&\mapsto \left(\boldsymbol{u}_{|t=0}-\psi,\partial_t \overline{\varrho}(\boldsymbol{u})-\lambda\frac{\overline{\varrho}(\boldsymbol{u})_{xx}}{\vert\overline{\varrho}(\boldsymbol{u})_x\vert ^2}\right)\,.
	\end{align*}
	Notice that $G(1,0,0)=0$.
	Due to $(D\varrho)_{|0}=0$ the Fr\'{e}chet derivative 
	\begin{equation*}
	(DG)_{|(1,0,0)}(0,0,\cdot):\boldsymbol{X}_T\to \boldsymbol{I}\times L_p\left((0,T)\times(0,1);(\mathbb{R}^n)^3\right)
	\end{equation*}
	is given by
	\begin{equation*}
	(DG)_{|(1,0,0)}(0,0,\boldsymbol{u})=\left(\boldsymbol{u}_{|t=0},\partial_t \boldsymbol{u}-\mathcal{A}_{T,\mathcal{E}}(\boldsymbol{u})\right)\,.
	\end{equation*}
	As explained in Remark~\ref{RemarkExistenceAnalysisLinearisationinEsigma} we have that  $(DG)_{|(1,0,0)}(0,0,\cdot)$ is an isomorphism. Hence the implicit function theorem implies the existence of neighbourhoods $\mathcal{U}$ of $(1,0)$ in $(1-\varepsilon,1+\varepsilon)\times\boldsymbol{I}$ and $\mathcal{V}$ of $0$ in $\boldsymbol{X}_T$ and a function $\zeta: \mathcal{U}\to\mathcal{V}$ with $\zeta((1,0))=0$ and
	\begin{equation*}
	\{(\lambda,\psi,\boldsymbol{u})\in \mathcal{U}\times\mathcal{V}:G(\lambda,\psi,\boldsymbol{u})=0\}=\{(\lambda,\psi,\zeta(\lambda,\psi)):(\lambda,\psi)\in\mathcal{U}\}\,.
	\end{equation*}
	Consider now the map $P:\boldsymbol{E}_T\to\boldsymbol{X}_T$ given by $P(\gamma):=P_{\boldsymbol{X}_T}(\gamma-\mathcal{E}\sigma)$ with $P_{\boldsymbol{X}_T}(\eta)=\boldsymbol{u}$ for the unique partition $\eta=\boldsymbol{u}+\overline{\boldsymbol{u}}\in \boldsymbol{X}_T\oplus\boldsymbol{Z}_T$. Clearly, we have that $\overline{\varrho}(P(\gamma))=\gamma$ for all $\gamma$ in the neighbourhood $V$ constructed in Lemma~\ref{LemmaParametrizationofNonLinearBoundaryConditions}. Given $\lambda$ close to $1$ we consider the time-scaled function
	\begin{equation*}
	(\mathcal{E}{\sigma})_
	{\lambda}(t,x):=(\mathcal{E}{\sigma})(\lambda t,x)\,.
	\end{equation*}
	By definition this satisfies for $\psi:=P((\mathcal{E}{\sigma})_\lambda)\big|_{t=0}$
	\begin{equation*}
	G(\lambda,\psi,P((\mathcal{E}{\sigma})_\lambda))=0\,.
	\end{equation*}
	By uniqueness we conclude that
	\begin{equation*}
	P((\mathcal{E}{\sigma})_{\lambda})=\zeta(\lambda,\psi)
	\end{equation*}
	and therefore
	\begin{equation*}
	(\mathcal{E}{\sigma})_{\lambda}=\bar{\varrho}(\zeta(\lambda,\psi))\,.
	\end{equation*}
	As both $\zeta$ and $\bar{\varrho}$ are smooth, this shows that $(\mathcal{E}{\sigma})_{\lambda}$ is a smooth function in $\lambda$ with values in $\boldsymbol{E}_T$. This implies now
	\begin{align}\label{EquationHilfsequionTimeRegularity}
	t\partial_t(\mathcal{E}{\sigma})=\partial_{\lambda}((\mathcal{E}{\sigma})_{\lambda})\big|_{\lambda=1}\in \boldsymbol{E}_T
	\end{align}
	from which we directly conclude~\eqref{EquationHigherTimeRegularity}.
\end{proof}
Next, we want to derive higher regularity in space for our solution. But this follows almost immediately from the associated ODE we have at a fixed time.
\begin{cor}[Higher space regularity of solutions to the Special Flow]\label{corollarhigherspaceregularity}\ \\
Let $\mathcal{E}\sigma\in\boldsymbol{E}_T$ be a solution to the Special Flow in $[0,T]$ with $T>0$ and initial value $\sigma\in W_p^{2-\nicefrac{2}{p}}\left((0,1);(\mathbb{R}^n)^3\right)$. 
Given $\tilde{t}\in(0,T]$ we have for almost all $t\in[\tilde{t},T]$ that 
	\begin{equation*}
	(\mathcal{E}\sigma)(t)\in W^{3}_{\nicefrac{p}{2}}((0,1);(\R^n)^3)\,.
	\end{equation*}	
	In particular, there is an $\alpha>0$ such that $(\mathcal{E}\sigma)(t)\in C^{2+\alpha}((0,1);(\R^n)^3)$ for almost all $t\in[\tilde{t},T]$.
\end{cor}
\begin{proof}
	Considering $\partial_t((\mathcal{E}\sigma)^i)(t)$ as given functions $\mathfrak{f}^i\in W^{1}_p((0,1);\R^n)$ we see that $(\mathcal{E}\sigma)^i(t,\cdot)$ solves
	\begin{align*}
	\frac{(\mathcal{E}\sigma)^i_{xx}(t,\cdot)}{\left\vert(\mathcal{E}\sigma)^i_x(t,\cdot)\right\vert^2}=\mathfrak{f}^i
	\end{align*}
	in $W^{1}_p((0,1);(\R^2)^3)$ for almost every $t\in[\tilde{t},T]$. As for almost every $t\in[\tilde{t},T]$ the function $(\mathcal{E}\sigma)^i(t,\cdot)$ is in $W^{2}_p((0,1);\R^n)$, we know that $\left\vert(\mathcal{E}\sigma)^i_x(t,\cdot)\right\vert^2\in W^{1,\nicefrac{p}{2}}((0,1);\R)$ for almost every $t\in[\tilde{t},T]$ and thus we can multiply the above equations to see that
	\begin{align*}
	(\mathcal{E}\sigma)^i_{xx}(t,\cdot)=\tilde{\mathfrak{f}}^i\in W^{1}_{\nicefrac{p}{2}}((0,1);\R^n)
	\end{align*}
	for almost every $t\in[\tilde{t},T]$ with new given inhomogeneities $\tilde{\mathfrak{f}}^i$. Note that we used here that with our choice of $p$ the Sobolev space $W^{1}_{\nicefrac{p}{2}}$ is indeed a Banach algebra on one dimensional domains. But from the last equation we directly conclude $(\mathcal{E}\sigma)^i\in W^{3}_{\nicefrac{p}{2}}((0,1);(\R^n)^3)$. The second claim is just a direct consequence of the Sobolev embeddings.
\end{proof}
With the two previous results we are now able to start a bootstrap procedure.
\begin{thm}[Smoothness of solutions to the Special Flow]\label{smoothnessthm}\ \\
	Let $\mathcal{E}\sigma\in\boldsymbol{E}_T$ be a solution to the Special Flow in $[0,T]$ with $T>0$ and initial value $\sigma\in W_p^{2-\nicefrac{2}{p}}\left((0,1);(\mathbb{R}^n)^3\right)$. Then $\mathcal{E}\sigma$ is smooth on $ [\tilde{t}, T]\times[0,1]$ for all  $\tilde{t}\in (0,T)$.
\end{thm}
\begin{proof}
	Due to Corollary \ref{corollarhigherspaceregularity} we can use $(\mathcal{E}\sigma)(t)$ for almost all $t>0$ as initial data for a regularity result in parabolic H\"older space,
	cf.~\cite{garckemenzelpluda2019willmore} for such a result for the Willmore flow. As we checked that the Lopatinskii-Shapiro conditions are still valid in higher co-dimensions, the analysis works as in the planar case. Additionally, the needed compatibility conditions due to the zero order boundary conditions are guaranteed by the fact that $\partial_t(\mathcal{E}\sigma)$ lies in $C([\tilde{t},T]; C([0,1];(\R^n)^3)$. With this new maximal regularity result, which is the key argument in the proof of 
	Proposition~\ref{propositionhighertimeregularityspecialflow}, we can repeat the whole procedure to derive $C^{3+\alpha, (3+\alpha)/2}$-regularity. This starts now the bootstrapping yielding the desired smoothness result. Note that in every step the needed compatibility conditions are guaranteed by the fact that our flow already has the regularity related to these compatibility conditions (see for instance~\cite[Theorem 3.1]{mannovtor}). 
\end{proof}

In analogy to~\cite{garckemenzelpludawillmore} we may now use smoothness of the Special Flow to prove Theorem~\ref{existenceuniqueness}.

\begin{proof}[Proof of Theorem~\ref{existenceuniqueness}]
		The existence of maximal solutions and their geometric uniqueness are shown in Proposition~\ref{exmax}. Using smoothness of the Special Flow shown in Theorem~\ref{smoothnessthm} one may argue analogously to~\cite[Section 5.2, Section 7.2]{garckemenzelpludawillmore} to show that parametrising each curve $\mathbb{T}^i(t)$ with constant speed equal to the length of $\mathbb{T}^i(t)$ yields a global parametrisation $\gamma:[0,T_{max})\times[0,1]\to(\mathbb{R}^n)^3$ of the evolution that is smooth for positive times.
	\end{proof}


\section{Long Time Behaviour of the Motion by Curvature}\label{longtime}


\begin{proof}[Proof of Theorem~\ref{longtimebehavior}]
Let $\varepsilon\in (0,\nicefrac{T_{\max}}{1000})$
be fixed.
Suppose that $T_{\max}$ is finite and that
the two assertions $i)$ and $ii)$ are not fulfilled.
Let $\gamma=(\gamma^1,\gamma^2,\gamma^3):[0,T_{max})\times[0,1]\to(\mathbb{R}^n)^3$ be the parametrisation of the evolution such that each curve $\mathbb{T}^i(t)$ is parametrised with constant speed equal to its length $L(\mathbb{T}^i(t))$. As $\gamma$ is smooth on $[\varepsilon,T]$ for all positive $\varepsilon$ and all $T\in (\varepsilon, T_{\max})$,
hypothesis $ii)$ yields 
\begin{equation*}
\boldsymbol{\kappa}^i\in L^\infty\left((\varepsilon,T_{\max}); L^2((0,1);\mathbb{R}^n)\right)\,.
\end{equation*}
As $\boldsymbol{E}_T$ embeds continuously into $C\left([0,T];C^1([0,1];(\mathbb{R}^n)^3)\right)$, hypothesis $i)$ implies that the lengths $L(\mathbb{T}^i)$ of all three curves
composing the network are uniformly bounded away from zero in $[0,T_{\max})$.
Moreover, thanks to the gradient flow structure of the motion by curvature
the single lengths of the networks satisfy $L(\mathbb{T}^i(t))\leq L(\mathbb{T}_0)$ for all 
$t\in [0,T_{\max})$. In particular, we obtain for all $t\in[0,T_{max})$, $x\in[0,1]$,
\begin{equation}\label{lengthboundedbelow}
0<\mathfrak{c}\leq  \vert \gamma^i_{x}(t,x)\vert=L(\mathbb{T}^i(t))\leq {C}<\infty\,.
\end{equation}
With our choice of parametrisation the curvature can be expressed as
$\boldsymbol{\kappa}^i=\nicefrac{\gamma^i_{xx}}{L(\mathbb{T}^i)^2}$ from which we can infer for all $t\in[0,T_{max})$,
$$
\int_{0}^1 \vert \gamma^i_{xx}\vert^2\,\mathrm{d}x=
\left(L(\mathbb{T}^i)\right)^3\int_{\mathbb{T}}\vert \boldsymbol{\kappa}^i\vert^2\,\mathrm{d}s
\leq C<\infty\,.
$$
As the endpoints $P^1$, $P^2$, $P^3$ are fixed and as the single lengths $L(\mathbb{T}^i(t))$ are uniformly bounded from above in $[0,T_{max})$, there exists a constant $R>0$ such that for every $t\in [0,T_{\max})$
it holds $\mathbb{T}(t)\subset B_R(0)$. With the above arguments we conclude
$$
\gamma^i\in L^{\infty}((\varepsilon, T_{\max}); W^2_2((0,1);\mathbb{R}^n))\,.
$$
The Sobolev Embedding Theorem implies for all $p\in(3,6]$ the estimate
\begin{equation}\label{estimatenorm}
\sup_{t\in (\varepsilon,T_{max})}\left\lVert\gamma^i(t)\right\rVert_{W^{2-\nicefrac{2}{p}}_p\left((0,1);\mathbb{R}^n\right)}\leq\boldsymbol{C}
\end{equation}
for a uniform constant $\boldsymbol{C}>0$.
We note further that for all $\delta\in\left(0,\nicefrac{T_{max}}{4}\right)$ the parametrisation
$\gamma(T_{max}-\delta)$ is an admissible initial value
for the Special Flow~\eqref{problema}. Due to~\eqref{lengthboundedbelow} and~\eqref{estimatenorm} Theorem~\ref{short time existence} yields that there exists a uniform time $\boldsymbol{T}$ of existence of solutions to the Special Flow~\eqref{problema} for all initial values $\gamma(T_{max}-\delta)$ depending on $\boldsymbol{C}$ and $\mathfrak{c}$. Let $\delta:=\min\left\{\nicefrac{\boldsymbol{T}}{2},\nicefrac{T_{max}}{4}\right\}$. Then Theorem~\ref{short time existence} implies the existence of a solution $\eta=(\eta^1,\eta^2,\eta^3)$ with $\eta^i$ regular and 
\begin{equation*}
\eta^i\in W_p^1
\left((T_{max}-\delta, T_{max}+\delta);L_p\left((0,1);\mathbb{R}^n\right)\right)
\cap L_p\left((T_{max}-\delta, T_{max}+\delta);W_p^2\left((0,1);\mathbb{R}^n\right)\right)
\end{equation*}
to system~\eqref{problema} with $\eta\left(T_{max}-\delta\right)=\gamma\left(T_{max}-\delta\right)$.
The two parametrisations $\gamma$ and $\eta$ defined on $(0,T_{max}-\frac{\delta}{3})$ and $\left(T_{max}-\frac{\delta}{2},T_{max}+\delta\right)$, respectively, define a solution $(\widetilde{\mathbb{T}}(t))$ to the motion by curvature on the time interval $(0,T_{max}+\delta]$ with initial datum $\mathbb{T}_0$ coinciding with $\mathbb{T}$ on $(0,T_{max})$. This contradicts the maximality of $T_{max}$.
\end{proof}

\newpage

\bibliographystyle{amsplain}
\bibliography{MCFSobolev}

\end{document}